\documentclass[a4paper,12pt]{article}
\usepackage[dvips]{graphicx}
\usepackage{amsmath}
\usepackage{hyperref}
\usepackage{amssymb}
\usepackage{amsfonts}
\usepackage{ucs}
\usepackage[utf8x]{inputenc}
\usepackage[T1]{fontenc}
\usepackage{url}
\usepackage{psfrag}
\usepackage{dsfont}
\usepackage{color}

\usepackage{float}
\usepackage{rotating}
\def\del  {\partial}
\def\eps{\varepsilon}

\def\R{\mathbb{R}}

\def\Div{\operatorname{div}}

\newtheorem{remark}{\textbf{Remark}}
\newtheorem{theorem}{\textbf{Theorem}}
\newtheorem{lemma}{\textbf{Lemma}}

\def\del  {\partial}
\def\eps{\varepsilon}

\def\R{\mathbb{R}}

\newenvironment{proof}{%
{\noindent \bf Proof : }%
}{%
\hfill$\Box$\\%
}

\begin{document}

\title {MODELS OF NONLINEAR ACOUSTICS VIEWED AS\\ AN APPROXIMATION OF THE KUZNETSOV EQUATION}

\author{Adrien Dekkers\footnote{Laboratory  Math\'ematiques et Informatique pour la Complexit\'e et les Syst\`emes, CentraleSup\'elec, Univ\'ersit\'e Paris-Saclay, Campus de Gif-sur-Yvette, Plateau de Moulon, 
3 rue Joliot Curie, 91190 Gif-sur-Yvette, France, 
(adrien.dekkers@centralesupelec.fr).}
, \quad Vladimir Khodygo\footnote{Department of Biological, Environmental and Rural Sciences,Aberystwyth University,
Penglais Campus,Aberystwyth, 
Ceredigion, 
SY23 3FL, UK, (vlk@aber.ac.uk)}
\quad           and \quad Anna Rozanova-Pierrat \footnote{Laboratory  Math\'ematiques et Informatique pour la Complexit\'e et les Syst\`emes, CentraleSup\'elec, Univ\'ersit\'e Paris-Saclay, Campus de Gif-sur-Yvette, Plateau de Moulon, 
3 rue Joliot Curie, 91190 Gif-sur-Yvette, France, 
(anna.rozanova-pierrat@centralesupelec.fr), https://www.saroan.fr/anna/cv/}}

\maketitle

\begin{abstract}
 We relate together different models of non linear acoustic in thermo-elastic media as the Kuznetsov equation, the Westervelt equation, the Khokhlov-Zabolotskaya-Kuznetsov (KZK) equation and the Nonlinear Progressive wave Equation (NPE) and  estimate the time during which the solutions of these models keep closed in the $L^2$ norm.  The KZK and NPE equations are  considered as paraxial approximations of the Kuznetsov equation. The Westervelt equation is obtained as a nonlinear approximation of the Kuznetsov equation. Aiming to compare the solutions of the exact and approximated systems in found approximation domains the
 well-posedness results (for the Kuznetsov equation in a half-space with periodic in time initial and boundary data)  are obtained.
\end{abstract}

\section{Introduction.}\label{intro}

One of the most general model to describe an
acoustic wave propagation in an homogeneous thermo-elastic medium is the compressible Navier-Stokes system in $\mathbb{R}^n$
\begin{align}
&\partial_t \rho+ \Div(\rho \mathbf{v})=0, \label{N-S1}\\ 
&\rho[\partial_t \mathbf{v}+(\mathbf{v}.\nabla) \mathbf{v}]=-\nabla p+\eta \Delta \mathbf{v}+\left(\zeta+\frac{\eta}{3} \right)\nabla.\Div (\mathbf{v}), \\
&\rho T [\partial_t S+(\mathbf{v}.\nabla)S]=\kappa \Delta T+\zeta (\Div \mathbf{v})^2 \nonumber\\
&+\frac{\eta}{2}\left( \partial_{x_k} v_i+\partial_{x_i}v_k-\frac{2}{3}\delta_{ik}\partial_{x_i} v_i\right)^2,\label{N-S3}\\
&p=p(\rho,S),\label{N-S4}
\end{align}
where
the pressure $p$ is given by the state law $p=p(\rho,S)$. The density $\rho$, the velocity $\mathbf{v}$, the temperature $T$ and the entropy $S$ are unknown functions in  system~(\ref{N-S1})--(\ref{N-S4}). The coefficients $\zeta,\;\kappa$ and $\eta$ are constant viscosity coefficients. For the acoustical framework the wave motion is supposed to be potential and the viscosity coefficients are supposed to be small in terms of a dimensionless small parameter $\eps>0$, which also characterizes  the size of the perturbations near
the constant state $(\rho_0, 0, S_0, T_0)$. Here  the velocity $\mathbf{v}_0$ is taken equal to $0$ just using a Galilean transformation.

Actually, $\eps$ is the Mach number, which is supposed to be small~\cite{Bakhvalov} ($\epsilon=10^{-5}$ for the propagation in water with an initial
power of the order of $0.3 \, \mathrm{W} /\mathrm{cm}^2$):
\begin{equation*}
  \dfrac{\rho-\rho_0}{\rho_0}\sim\dfrac{T-T_0}{T_0}\sim\dfrac{|\mathbf{v}|}{c}\sim\epsilon,
 \end{equation*}
where  $c = \sqrt{p' (\rho_0)}$  is the  speed  of  sound  in  the  unperturbed  media.

Hence as in~\cite{Dekkers,Roz1}, system~(\ref{N-S1})--(\ref{N-S4}) becomes an isentropic Navier-Stokes system
\begin{gather}
\partial_t \rho_\varepsilon  +\operatorname{div}( \rho_\varepsilon  \mathbf{v}_\varepsilon)=0\,,\label{NSi1}\\
 \rho_\varepsilon[\partial_t
\mathbf{v}_\varepsilon+(\mathbf{v}_\varepsilon \cdot\nabla) \,
\mathbf{v}_\varepsilon] = -\nabla p(\rho_\varepsilon)
+\varepsilon \nu \Delta \mathbf{v}_\varepsilon\,,\label{NSi2}
\end{gather}
with the approximate state equation $p(\rho, S)=p(\rho_\eps)+O(\eps^3)$:
\begin{equation}\label{press}
p(\rho_{\varepsilon})
=p_0+c^2(\rho_{\varepsilon}-\rho_0)+\frac{(\gamma-1)c^2}{2\rho_0} (\rho_{\varepsilon}-\rho_0)^2,
\end{equation}
where $\gamma=C_p/C_V$ denotes the ratio of the heat capacities at constant pressure and at constant volume respectively
and with a small enough and positive viscosity coefficient:
$$\varepsilon \nu=\beta+\kappa \left(\frac{1}{C_V}-\frac{1}{C_p}\right).$$
If we go on physical assumptions of the wave motion~\cite{Bakhvalov,Hamilton,Kuznetsov,Westervelt} for the perturbations of the density or of the velocity or of the pressure, the isentropic system~(\ref{NSi1})--(\ref{NSi2}) gives
\begin{enumerate}
 \item the Westervelt equation for the potential of the velocity, derived initially by Westervelt~\cite{Westervelt} and later by other authors~\cite{Aanonsen,TJO}:
 \begin{equation}\label{West}
 \partial_t^2\Pi-c^2\Delta \Pi=\varepsilon
\partial_t\left(\frac{\nu}{\rho_0}\Delta \Pi+\frac{\gamma+1}{2c^2} (\partial_t
\Pi)^2\right)
\end{equation}
with the same constants introduced for the Navier-Stokes system.
 \item the Kuznetsov equation also for the potential of the velocity, firstly introduced by Kuznetsov~\cite{Kuznetsov} for the velocity potential, see also Refs.~\cite{Hamilton,Jordan,Barbara,Lesser} for other different methods of its derivation:
 \begin{equation}\label{KuzEq}
  \partial^2_t u -c^2 \triangle u=\varepsilon\partial_t\left( (\nabla u)^2+\frac{\gamma-1}{2c^2}(\partial_t u)^2+\frac{\nu}{\rho_0}\Delta u\right).
 \end{equation}
 \item the Khokhlov-Zabolotskaya-Kuznetsov (KZK)~\cite{Bakhvalov,Roz3} for the density:
 \begin{equation}\label{KZKI}
 c\partial^2_{\tau z} I -\frac{(\gamma+1)}{4\rho_0}\partial_\tau^2
 I^2-\frac{\nu}{2 c^2\rho_0}\partial^3_\tau I-\frac{c^2}2 \Delta_y
 I=0.
 \end{equation}
 \item the Nonlinear Progressive wave Equation (NPE) derived in Ref.~\cite{McDonald} also for the density:
 \begin{equation}\label{NPE2}
 \partial^2_{\tau z} \xi+\frac{(\gamma+1)c}{4\rho_0}\partial_z^2[(\xi)^2]-\frac{\nu}{2\rho_0}\partial^3_z  \xi+\frac{c}{2}\Delta_y  \xi=0.
 \end{equation}
\end{enumerate}
For higher order models as the nonlinear Jordan-Moore-Gibson-Thompson (JMGT)
equation, containing the Kuznetsov equation as a particular or a limit case,
see~\cite{Jordan-2014,Kaltenbacher1,Kaltenbacher2} and their references. In this
article we don't consider such higher order models and focus our attention on
the Kuznetsov equation considered here as the most complete
equation.

In~\cite{Dekkers} it is shown that the Kuznetsov equation comes from the Navier-Stokes or Euler system only by small perturbations, but to obtain the KZK and the NPE equations
we also need to perform in addition to the small perturbations a paraxial change of variables. In this article we derive the KZK and the NPE equations  from the Kuznetsov equation  just performing the corresponding paraxial change of variables and show that the Westervelt equation can be also viewed as  an approximation of the Kuznetsov equation by a  nonlinear perturbation.

The physical context and the physical usage of the KZK and the NPE equations are different: the
NPE equation is helpful to describe short-time pulses and a long-range
propagation, for instance, in an ocean wave-guide, where the refraction
phenomena are important~\cite{Kuperman2,Kuperman1}, while the KZK equation
typically models the  ultrasonic propagation with strong diffraction phenomena,
combining with finite amplitude effects (see~\cite{Roz3} and the references
therein).
But in the same time~\cite{Dekkers},
 there is a bijection
between the variables of these two models and they can be presented by the same type differential operator with  constant positive coefficients:
$$Lu=0, \quad L=\partial^2_{t x}-c_1\partial_x(\partial_x\cdot )^2-c_2\partial^3_x \pm c_3\Delta_y \quad \hbox{for } t\in \R^+,\; x\in\R,\; y\in \R^{n-1} .$$
Therefore, the results on the solutions of the KZK equation
from~\cite{Ito,Roz2}
are valid for the NPE equation.

The interest to study how closed are the solutions of the general model of the non-linear wave motion, described by the Kuznetsov equation, and of  simplified
models with more particular area of application (such as the KZK equation and the NPE
equation which are valid only with additional assumptions on the wave propagation
describing by the paraxial changes of variables) is naturally motivated by the questions about the accuracy of the approximations and of a comparative analysis of
the solutions of these models.

If we formally consider the differential operators in the Kuznetsov equation~(\ref{KuzEq}) and the Westervelt equation~(\ref{West}), we notice that the Westervelt equation
keeps only one of two non-linear terms of the Kuznetsov equation, producing cumulative effects in a progressive wave propagation~\cite{Aanonsen}.
The  question  how closed are the solutions of these two
models, which differ on presence  of local nonlinear term,
was also open so far.

The mentioned approximation questions are treated theoretically in this article.

Let us also notice that the Kuznetsov equation~(\ref{KuzEq}) (and also the Westervelt equation~(\ref{West})) is a non-linear wave equation with terms of different order  (the wave operator is of order $\eps^0$ and the nonlinear and viscosity terms are of  order $\eps$). But the KZK- and NPE-paraxial approximations allow to have the  approximate equations with all terms of the same order, $i.e.$ the KZK and NPE equations.
For the well posedness of the Cauchy problem for the Kuznetsov equation we
cite~\cite{Perso} and for boundary value problems in regular bounded domains
see~\cite{Kalt2,Kalt1,Meyer}.

We present the structure of the paper and its mains results in the next subsection.

\subsection{Main results}\label{IntroSub}

 To keep a physical sense of the approximation problems, we consider especially the two or three dimensional cases, $i.e.$ $\R^n$ with $n=2$ or $3$, and in the following we use the notation
 $x=(x_1,x')\in \R^n$ with one propagative axis $x_1\in \R$ and the traversal variable $x'\in \R^{n-1}$.

\renewcommand{\arraystretch}{2.0}

%

\begin{sidewaystable}[ph!]
\vspace{13cm}\centering \caption{Approximation results for models derived from the Kuznetsov equation}\label{TABLE2}
\begin{tabular}{@{} c | c | c || c || c | c | }
\cline{2-6}
    & \multicolumn{2}{| c ||}{\quad KZK \quad}
    &\quad NPE\quad
    & \multicolumn{2}{| c |}{\quad  Westervelt\quad}
\\
\cline{2-6}
    &\shortstack{periodic\\
    boundary condition\\
    problem}
    & \shortstack{initial  \\
    boundary value\\
    problem}
    & \shortstack{viscous\\
    and inviscid\\
    case}
    & viscous case
    & inviscid case
    \\
    \hline
     \multicolumn{1}{| c |}{Theorem}
     &  Theorem~\ref{AproxKuzKZK}
    &  Theorem~\ref{approxKuzKZKbis}
    &  Theorem~\ref{approxKuzNPE}
    & \multicolumn{2}{| c |}{Theorem~\ref{ApproxKuzWes}}
     \\
    \hline
    \multicolumn{1}{| c |}{Derivation}
    &\multicolumn{2}{| c ||}{\shortstack{ paraxial approximation\\
    $u=\Phi(t-\frac{x_1}{c},\varepsilon x_1,\sqrt{\varepsilon} \textbf{x}')$} }
    & \shortstack{ paraxial approximation\\
    $u=\Psi(\varepsilon t,x_1-ct,\sqrt{\varepsilon} \textbf{x}')$}
    &\multicolumn{2}{| c |}{ $\Pi=u+\frac{1}{c^2}\varepsilon u \partial_t u$}
    \\
    \hline
     \multicolumn{1}{| c |}{\shortstack{Approxi-\\mation\\
     domain}}
     & \multicolumn{2}{| c ||}{\shortstack{the half space \\$\{x_1>0,
x'\in \R^{n-1}\}$} }
& $\mathbb{T}_{x_1}\times\mathbb{R}^2$
& \multicolumn{2}{| c |}{$\mathbb{R}^n$}
 \\
    \hline
     \multicolumn{1}{| c |}{\shortstack{Approxi-\\mation \\
     order}}
     & \multicolumn{2}{| c ||}{$ O(\varepsilon)$}
     & $O(\varepsilon)$
     & \multicolumn{2}{| c |}{$O(\varepsilon^2)$ }
      \\
    \hline
    \multicolumn{1}{| c |}{Estimation}
    & \shortstack{ $\Vert I-I_{aprox}\Vert_{L^2(\mathbb{T}_t\times\mathbb{R}^{n-1})}\leq \varepsilon $\\
    $z\leq K$}
    &\shortstack{ $\Vert (u -\overline{u})_t(t)\Vert_{L^2}$\\$+ \Vert \nabla (u-\overline{u})(t)\Vert_{L^2}$\\$\leq K \varepsilon.$\\
    $t<\frac{T}{\varepsilon} $}
    & \shortstack{ $\Vert (u -\overline{u})_t(t)\Vert_{L^2}$\\$+ \Vert \nabla (u-\overline{u})(t)\Vert_{L^2}$\\
    $\leq K \varepsilon$\\
    $t<\frac{T}{\varepsilon} $}
    & \multicolumn{2}{| c |}{\shortstack{ $\Vert (u -\overline{u})_t(t)\Vert_{L^2}$\\
    $+ \Vert \nabla (u-\overline{u})(t)\Vert_{L^2}$\\
    $\leq K \varepsilon$\\
   $t<\frac{T}{\varepsilon} $ }}
    \\
    \hline
    \multicolumn{1}{| c
|}{\shortstack{Initial\\
    data\\
    regularity}}
    & \shortstack{$ I_0\in H^{s+\frac{3}{2}}(\mathbb{T}_{t}\times \mathbb{R}^{n-1}_{x'})$\\ for $s > \max(\frac{n}{2},2)$}
    & \shortstack{$ I_0\in H^{s}(\mathbb{T}_{t}\times \mathbb{R}^{n-1}_{x'})$\\
    for $\left[\frac{s}{2}\right]>\frac{n}{2}+2 $}
    &\shortstack{$ \xi _0\in H^{s+2}(\mathbb{T}_{x_1}\times \mathbb{R}^{n-1}_{x'})$\\ for $s>\frac{n}{2}+1$}
    &\shortstack{$u_0\in H^{s+3}(\mathbb{R}^n)$\\ $u_1\in H^{s+3}(\mathbb{R}^3)$\\
    for $s>\frac{n}{2}$}
    & \shortstack{$u_0\in H^{s+3}(\mathbb{R}^n)$\\ $u_1\in H^{s+2}(\mathbb{R}^3)$\\
    for $s>\frac{n}{2}$}
    \\
    \hline
    \multicolumn{1}{|c|}{\shortstack{Data\\regularity\\
    for remainder\\
    boundness}}
    & \shortstack{$ I_0\in H^{s+\frac{3}{2}}(\mathbb{T}_{t}\times \mathbb{R}^{n-1}_{x'})$\\ for $s > \max(\frac{n}{2},2)$}
    & \shortstack{$ I_0\in H^{6}(\mathbb{T}_{t}\times \mathbb{R}^{n-1}_{x'})$\\
    for $n= 2,3$,\\
    $ I_0\in H^{s}(\mathbb{T}_{t}\times \mathbb{R}^{n-1}_{x'})$\\
    for $\left[\frac{s}{2}\right]>\frac{n}{2}+1 $, $n\geq 4$}
    & \shortstack{$ \xi _0\in H^{4}(\mathbb{T}_{x_1}\times \mathbb{R}^{n-1}_{x'})$\\
    for $n=2,3$.\\
    $ \xi _0\in H^{s}(\mathbb{T}_{x_1}\times \mathbb{R}^{n-1}_{x'})$\\ for $s>\frac{n}{2}+2$, $n\geq4$.}
  &\shortstack{$u_0\in H^{s+3}(\mathbb{R}^n)$\\ $u_1\in H^{s+3}(\mathbb{R}^n)$\\
    for $s>\frac{n}{2}$}
    & \shortstack{$u_0\in H^{s+3}(\mathbb{R}^n)$\\ $u_1\in H^{s+2}(\mathbb{R}^n)$\\
    for $s>\frac{n}{2}$}
    \\
    \hline
    \end{tabular}
\end{sidewaystable}

To be able to consider the approximation of the Kuznetsov equation by the KZK equation (see Section~\ref{secKuzKZK}), we  establish (see Theorems~\ref{globwellposKuzper}
 and~\ref{ThMainWPnuPGlobHalf} in Appendix~\ref{secWPresults}) global well posedness results for the Kuznetsov equation in the half space similar to the previous framework for the KZK and the Navier-Stokes system considered in~\cite{Dekkers}.
Theorem~\ref{globwellposKuzper} corresponds to the well posedness of the
periodic in time Dirichlet boundary valued problem for the Kuznetsov equation in
the half space $\R^+\times\R^{n-1}$ (see Eq.~(\ref{kuzper})) for small enough
boundary data. In this case the boundary condition is considered as the initial
condition of the corresponding Cauchy problem in $\R^n$. The proof is based on
the maximal regularity result for the corresponding linear problem given in
Theorem~\ref{ThCelik} and on the application of a result of the nonlinear
functional analysis from~\cite[1.5.~Cor.,~p.~368]{Sukhinin}
 (see also~\cite[Thm.~4.2]{Perso}). We also applied it
to prove (see Theorem~\ref{ThMainWPnuPGlobHalf}) the well posedness for the
initial boundary valued problem for the Kuznetsov equation in the half
space~(\ref{kuzhalfinitbound}), once again combining with the maximal regularity
result for the linear problem (see the proof of Lemma~\ref{maxregstrdamphalf}).

 In Subsection~\ref{secderKuzKZK} we derive the KZK equation from the Kuznetsov equation by introducing the paraxial change of variables~(\ref{paraxpotnpe}).

 For the approximation framework for their solutions  we study two cases. The
first case  is treated in Sub-subsection~\ref{sssKKZKperbc} and it considers the
purely time periodic boundary problem in the \textit{ansatz} variables
$(z,\tau,y)$ moving with the wave, where we use the well-posedness result of
Theorem~\ref{globwellposKuzper}.  In this case the only  viscous medium can be
considered as the condition to be periodic in time is not compatible with shock
formations providing the loss of the regularity which may occur in the inviscid
medium (see~\cite[Thm.~1.3]{Roz2}). The approximation
results are formulated in Theorem~\ref{AproxKuzKZK}.

 The second case (see Sub-subsection~\ref{sssKKZKperbcIn}) studies the initial boundary-value problem for the Kuznetsov equation in the initial variables $(t,x_1,x')$ with data coming from the solution of the KZK equation, using this time the well posedness  results of Theorem~\ref{ThMainWPnuPGlobHalf}. This time we have the approximation results  for the viscous and inviscid cases (see Theorem~\ref{approxKuzKZKbis} and Remark~\ref{RemInvCKZK}).

 In Section~\ref{secKuzNPE} we establish the approximation result between the Kuznetsov equation and NPE equation  in the viscous and inviscid cases (see Theorem~\ref{approxKuzNPE}).

Finally  in Section~\ref{SecWest} we compare the solutions of the Westervelt and the Kuznetsov equations. We derive the Westervelt equation from the Kuznetsov equation by a nonlinear change of variables in Subsection~\ref{subSecWestD} and we validate the approximation in Subsection~\ref{subSecWestA} (see Theorem~\ref{ApproxKuzWes} for viscous and inviscid cases).

  We  denote by $u$ a solution of the ``exact'' problem for the Kuznetsov equation
 $Exact(u)=0$
 and by $\overline{u}$ an approximate  solution, constructed by the derivation \textit{ansatz} from a regular solution of one of the approximate models (for instance of the KZK  or of the NPE  equations), $i.e.$  $\overline{u}$ is a function which solves the Kuznetsov equation up to $\eps$ terms, denoted by $\eps R$: $$Approx(\overline{u})=Exact(\overline{u})-\eps R=0.$$
 In the approximation between the solutions of the Kuznetsov equation and of the
Westervelt equation the remainder term appears with the size $\eps^2$ (it is
natural since both models contain terms of order $\eps^0$ and $\eps$).

 We can summarize  the obtained  approximation results of the Kuznetsov equation in the following way:
 if, once again, $u$ is a solution of the Kuznetsov equation and $\overline{u}$ is a solution of the NPE or of the KZK (for the initial boundary value problem) or of the Westervelt equations found for rather closed initial data $$\|\nabla_{t,\mathbf{x}} (u(0)-\overline{u}(0))\|_{L^2(\Omega)}\le \delta\le \eps,$$
then there exist constants $K$, $C_1$, $C_2$, $C>0$ independent of $\eps$, $\delta$ and on time, such that for all $t\le \frac{C}{\eps}$ it holds
$$\|\nabla_{t,\mathbf{x}} (u-\overline{u})\|_{L^2(\Omega)}\le C_1(\eps^2t+\delta)e^{C_2\eps t}\le K\eps.$$
For a more detailed  comparison between different models we include the main points of our results to the comparative  Table~\ref{TABLE2}.

 In Table~\ref{TABLE2} the line named ``Initial data regularity'' gives the information about the regularity of the initial data for the approximate model, which ensure the same regularity of the solutions of an approximate model and of the solution of the Kuznetsov equation, taken with the same initial data $u(0)=\overline{u}(0)$, coming from the corresponding \textit{ansatz}.

 To have the remainder term $R\in C([0,T],L^2(\Omega))$ we ensure that $Exact(\overline{u}) \in C([0,T],L^2(\Omega))$, $i.e.$ we need a sufficiently regular solution $\overline{u}$. The minimal regularity of the initial data to have a such $\overline{u}$ is given in Table~\ref{TABLE2} in the last line named ``Data regularity for remainder boundness''.

 To summarize, the rest of the paper is organized as follows.
Section~\ref{secKuzKZK} considers the derivation
(Subsection~\ref{secderKuzKZK}) of the KZK equation from the Kuznetsov equation
and two types of approximation results for the solutions of the Kuznetsov
equation approximated by the solutions of the KZK equation in
Subsection~\ref{secValKuzKZK}. The approximation by the solutions of the NPE
equation is considered in Section~\ref{secKuzNPE}. Section~\ref{SecWest}
contains the derivation of the Westervelt equation and the approximation result
for the solutions of the Kuznetsov and the Westervelt equations. The
well posedness results for the Kuznetsov equation needed for the approximation
results of Section~\ref{secKuzKZK} are detailed in Appendix~\ref{secWPresults}.

\section{The Kuznetsov equation and the KZK equation.}
\label{secKuzKZK}
\subsection{Derivation of the KZK equation from the Kuznetsov equation.}\label{secderKuzKZK}
If  the velocity potential is given~\cite{Kuznetsov} by \begin{equation}\label{paraxpot}
u(x,t)=\Phi(t-x_1/c,\epsilon x_1,\sqrt{\epsilon}x')=\Phi(\tau,z,y),
\end{equation}
 we directly obtain from the Kuznetsov equation~(\ref{KuzEq})  via the paraxial change of variables
\begin{equation}\label{chvarkzk}
\tau= t-\frac{x_1}{c},\;\;\;z=\varepsilon x_1,\;\;\; y=\sqrt{\varepsilon}x',
\end{equation}
  that
\begin{align}
& \partial^2_t u -c^2 \Delta u-\varepsilon\partial_t\left( (\nabla u)^2+\frac{\gamma-1}{2c^2}(\partial_t u)^2+\frac{\nu}{\rho_0}\Delta u\right)\nonumber\\
&=\varepsilon\left[2c\partial_{\tau z}^2\Phi-\frac{\gamma+1}{2c^2}\partial_{\tau}(\partial_{\tau} \Phi)^2-\frac{\nu}{\rho_0c^2}\partial^3_{\tau}\Phi- c^2 \Delta_y \Phi\right]+\varepsilon^2 R_{Kuz-KZK}\label{KZKpoten}\end{align}
with
{\small\begin{align}
\varepsilon^2 R_{Kuz-KZK}= &\varepsilon^2 \left(-c^2 \partial^2_z \Phi+\frac{2}{c} \partial_{\tau}(\partial_{\tau}\Phi \partial_z \Phi)-\partial_{\tau}(\nabla_y\Phi)^2+\frac{2\nu}{c\rho_0}\partial^2_{\tau}\partial_z \Phi-\frac{\nu}{\rho_0}\partial_{\tau} \Delta_y \Phi\right)\nonumber\\
&+\varepsilon^3\left(-\partial_{\tau}(\partial_z\Phi)^2-\frac{\nu}{\rho_0}\partial_{\tau}\partial^2_z\Phi\right).\label{remkuzkzk}
\end{align}}Let us notice that the paraxial change of variables~(\ref{chvarkzk}) defines the axis of the propagation $x_1$ along which the wave changes its profile much slower than along the transversal axis $x'$. This
is typical for the propagation of ultrasound waves.

Therefore,
we find that the right-hand side $\epsilon$-order terms in Eq.~(\ref{KZKpoten}) is exactly the KZK equation~(\ref{KZKI}).
Thanks to~\cite[Thms.~1.1--1.3]{Roz2} we have the  well posedness result for the KZK equation in the half space with periodic boundary conditions of a period $L$ on $\tau$ and of mean value zero.

Due to the well posedness domain $(\mathbb{T}_\tau\times \mathbb{R}^{n-1})$ of the KZK equation, to validate the approximation between the solutions of the KZK and the Kuznetsov equations, we need to have the well posedness of the Kuznetsov equation on the half space with boundary conditions coming from the initial condition for the KZK equation. For these well posedness results see Appendix~\ref{secWPresults}.

\subsection{Approximation of the solutions of the Kuznetsov equation by the solutions of the KZK equation.}\label{secValKuzKZK}
Let us  consider  the Cauchy problem associated  with the KZK
equation
\begin{equation}\label{NPEcau2}
 \left\lbrace
 \begin{array}{c}
 c\partial_{ z} I -\frac{(\gamma+1)}{4\rho_0}\partial_\tau
 I^2-\frac{\nu}{2 c^2\rho_0}\partial^2_\tau I-\frac{c^2}2 \partial_{\tau}^{-1}\Delta_y
 I=0\hbox{ on }\mathbb{T}_{\tau}\times\mathbb{R}_+\times\mathbb{R}^{n-1},\\
 I(\tau,0,y)=I_0(\tau,y)\hbox{ on }\mathbb{T}_{\tau}\times\mathbb{R}^{n-1},
 \end{array}\right.
 \end{equation}
for small enough initial data in order to have
by~\cite[Thm.~1.2]{Roz2} a time periodic solution $I$ defined on
$\mathbb{R}_+\times\mathbb{R}^{n-1}$. As it was  mentioned in
Introduction~\ref{IntroSub}, if $\nu>0$, to compare the solutions of the
Kuznetsov and the KZK equations we consider two cases.  The first case   (see
Sub-subsection~\ref{sssKKZKperbc}) consists in studies of the time periodic
boundary problem for the Kuznetsov equation~(\ref{kuzper}) with the boundary
condition imposed by the initial condition $I_0$ of  the KZK equation.
In Sub-subsection~\ref{sssKKZKperbcIn} we study the second case, when  the
solution of the KZK equation, taken for $\tau=0$, gives $I(0,z,y) $ defined on
$\mathbb{R}_+\times\mathbb{R}^{n-1}$, from  which we deduce,
according to the derivation \textit{ansatz}, both an initial condition for the
Kuznetsov equation at $t=0$ and a corresponding boundary condition.
In this second situation, it also makes sense to consider the inviscid case, briefly commented in the end of
Sub-subsection~\ref{sssKKZKperbcIn}.
\subsubsection{Approximation problem for the Kuznetsov with periodic boundary conditions.}\label{sssKKZKperbc}

By~\cite[Thm.~1.2]{Roz2}   there is a unique solution
$I(\tau,z,y)$ of the Cauchy problem for the KZK equation~(\ref{NPEcau2}) such
that
\begin{equation}\label{regsolKZK}
z\mapsto I(\tau,z,y) \in C([0,\infty[,H^s(\Omega_1))
\end{equation}
with $\int_{\mathbb{T}_{\tau}} I(l,z,y) dl=0$ and $\Omega_1=\mathbb{T}_{\tau}\times \mathbb{R}^{n-1}$, where $\mathbb{T}_{\tau}$ represents the periodicity in $\tau$ of period $L$.
 The operator $\partial_\tau^{-1}$ is  defined by
\begin{equation}\label{invdtau}
 \partial_{\tau}^{-1} I(\tau,z,y):=\int_0^{\tau} I(\ell,z,y) d\ell+\int_0^L \frac{\ell}{L}I(\ell,z,y)d\ell.
 \end{equation}
Formula~(\ref{invdtau}), which implies that $ \partial_{\tau}^{-1} I$ is $L$-periodic in $\tau$ and of mean value zero, gives us the estimate
\begin{equation}\label{EqEstimateAdrien}
 \Vert \partial_{\tau}^{-1} I\Vert_{H^s(\Omega_1)} \leq C \Vert
\partial_{\tau} \partial_{\tau}^{-1} I\Vert_{H^s(\Omega_1)} = C \Vert
I\Vert_{H^s(\Omega_1)}.
\end{equation}
So $\partial_{\tau}^{-1} I\vert_{z=0}\in H^s(\Omega_1)$, and hence
by~(\ref{regsolKZK})
$$
z\mapsto \partial_{\tau}^{-1} I(\tau,z,y) \in C([0,\infty[,H^s(\Omega_1)),
$$
with $\int_{\mathbb{T}_{\tau}}\partial_{\tau}^{-1} I(s,z,y) ds=0$.

We define on $\mathbb{T}_t\times \mathbb{R}_+\times \mathbb{R}^{n-1}$ 
\begin{equation}\label{potentialkzk}
\overline{u}(t,x_1,x'):=\frac{c^2}{\rho_0}\partial_{\tau}^{-1} I(\tau,z,y)=\frac{c^2}{\rho_0}\partial_{\tau}^{-1} I\left(t-\frac{x_1}{c},\varepsilon x_1,\sqrt{\varepsilon}x'\right)
\end{equation}
 with the paraxial change of variable~(\ref{chvarkzk}) associated
 with the KZK equation. Thus $\overline{u}$ is $L$-periodic in time
and of mean value zero.
Now we consider the following Kuznetsov problem
\begin{equation}\label{kuzper}
\left\lbrace
\begin{array}{l}
u_{tt}-c^2 \Delta u-\nu \varepsilon \Delta u_t=\alpha \varepsilon u_t\; u_{tt}+\beta \varepsilon \nabla u\; \nabla u_t\;\;\;\;\hbox{ on }\; \mathbb{T}_t\times \mathbb{R}_+\times\mathbb{R}^{n-1},\\
u\vert_{x_1=0}=g\;\;\;\;\hbox{ on }\; \mathbb{T}_t\times \mathbb{R}^{n-1},
\end{array}\right.
\end{equation}
in which  the boundary condition is imposed by the initial condition for the KZK equation:
\begin{equation}\label{boundcondkuz}
g(t,x'):=\overline{u}(t,0,x')=\frac{c^2}{\rho_0}\partial_{\tau}^{-1} I_0(\tau,y).
\end{equation}
Let us define (see Eq.~(\ref{paraxpot}),
  and subsection~4.1 in~\cite{Dekkers} for more details)
\begin{equation}\label{tildI}
\tilde{I}:=\frac{\rho_0}{c^2}\partial_{\tau}\Phi.
\end{equation}
 Then $\tilde{I}$ is the solution of the Kuznetsov equation written in the following form with the remainder $R_{Kuz-KZK}$ defined in Eq.~(\ref{remkuzkzk}):
\begin{equation}\label{KZKwithre}
\left\lbrace
\begin{array}{l}
c\partial_{ z} \tilde{I} -\frac{(\gamma+1)}{4\rho_0}\partial_\tau
\tilde{I}^2-\frac{\nu}{2 c^2\rho_0}\partial^2_\tau \tilde{I}-\frac{c^2}2 \Delta_y\partial_{\tau}^{-1}
\tilde{I}+\varepsilon \frac{\rho_0}{2c^2}R_{Kuz-KZK}=0,\\
\tilde{I}\vert_{z=0}=I_0.
\end{array}
\right.
\end{equation}
In Eq.~(\ref{KZKwithre}) we can recognize the system associated
 with the KZK equation~(\ref{KZKI}).

Now we can formulate the following approximation result between the solutions of the KZK and Kuznetsov equations.
\begin{theorem}\label{AproxKuzKZK}
Let $\nu>0$.
For $s>\max(\frac{n}{2},2)$ and $I_0\in H^{s+\frac{3}{2}}(\mathbb{T}_{\tau}\times \mathbb{R}^{n-1})$ small enough in $H^{s+\frac{3}{2}}(\mathbb{T}_{\tau}\times \mathbb{R}^{n-1}) $, there exists a unique global solution $I$ of the Cauchy problem for the KZK equation~(\ref{NPEcau2}) such that
$$z\mapsto I(\tau,z,y) \in C([0,\infty[,H^{s+\frac{3}{2}}(\mathbb{T}_{\tau}\times \mathbb{R}^{n-1})).$$
In addition, there exists a unique global solution $\tilde{I}$ of the Kuznetsov problem~(\ref{KZKwithre}), in the sense
$\tilde{I}:=\frac{\rho_0}{c^2}\partial_{\tau}\Phi,$
with $\Phi(\tau,z,y):=u(t,x_1,x')$ with the paraxial change of variable~(\ref{chvarkzk}) and
$$u\in H^2(\mathbb{T}_t;H^s(\mathbb{R}^+\times\mathbb{R}^{n-1}))\cap H^1(\mathbb{T}_t;H^{s+2}(\mathbb{R}^+\times\mathbb{R}^{n-1})),$$
is the global solution of the periodic problem~(\ref{kuzper}) for the Kuznetsov equation with $g$ defined by $I_0$ as in Eq.~(\ref{boundcondkuz}). Moreover there exist $C_1$,  $C_2>0$ such that
$$\frac{1}{2}\frac{d}{dz}\Vert I-\tilde{I}\Vert^2_{L^{2}(\mathbb{T}_\tau\times \mathbb{R}^{n-1})}\leq C_1 \Vert I-\tilde{I}\Vert^2_{L^{2}(\mathbb{T}_\tau\times \mathbb{R}^{n-1})}+C_2\varepsilon \Vert I-\tilde{I}\Vert_{L^{2}(\mathbb{T}_\tau\times \mathbb{R}^{n-1})},$$
which implies
$$\Vert I-\tilde{I}\Vert_{L^{2}(\mathbb{T}_\tau\times \mathbb{R}^{n-1})}(z)
\leq \frac{C_2}{2}\eps z e^{\frac{C_1}{2} z}
\leq \frac{C_2}{C_1}\varepsilon(e^{\frac{C_1}{2}z}-1)$$
and $$\Vert I-\tilde{I}\Vert_{L^{2}(\mathbb{T}_\tau\times \mathbb{R}^{n-1})}(z)\leq K\varepsilon \hbox{ while } z\leq C$$ with $K>0$ and $C>0$ independent of $\varepsilon$.
\end{theorem}
\begin{proof}
For $s> \max(\frac{n}{2},2)$, the global well-posedness of $I$ comes
from~\cite[Thm.~1.2]{Roz2} if $I_0\in
H^{s+\frac{3}{2}}(\mathbb{T}_{\tau}\times \mathbb{R}^{n-1})$ is small enough.
Moreover, since $g$ is given by Eq.~(\ref{boundcondkuz}), thanks to the
definition  of $\partial_{\tau}^{-1}$ in~(\ref{invdtau}) and the fact that
$I_0\in H^{s+\frac{3}{2}}(\mathbb{T}_{\tau}\times \mathbb{R}^{n-1})$, we have
$$g\in H^{s+\frac{3}{2}}(\mathbb{T}_t\times \mathbb{R}^{n-1}) \hbox{ and }\partial_t g\in H^{s+\frac{3}{2}}(\mathbb{T}_t\times \mathbb{R}^{n-1}).$$
And thus
$$g\in H^{\frac{7}{4}}(\mathbb{T}_t;H^s(\mathbb{R}^{n-1}))\cap H^1(\mathbb{T}_t;H^{s+2-\frac{1}{2}}(\mathbb{R}^{n-1})).$$
Therefore we can use Theorem~\ref{globwellposKuzper}, which implies the global existence of the periodic in time solution
\begin{equation}\label{EqUTh21}
 u\in H^2(\mathbb{T}_t;H^s(\mathbb{R}^+\times\mathbb{R}^{n-1}))\cap H^1(\mathbb{T}_t;H^{s+2}(\mathbb{R}^+\times\mathbb{R}^{n-1})),
 \end{equation}
 of the Kuznetsov periodic boundary value problem~(\ref{kuzper}) as $I_0$ is small enough in $H^{s+\frac{3}{2}}(\mathbb{T}_{\tau}\times \mathbb{R}^{n-1})$.
Therefore, it also implies the global existence of $\tilde{I}$, defined in~(\ref{tildI}), which is the solution of the exact Kuznetsov system~(\ref{KZKwithre}).

Now we subtract the equations in  systems~(\ref{NPEcau2}) and~(\ref{KZKwithre}) to obtain
\begin{align*}
c\partial_z (I-\tilde{I})-\frac{\gamma+1}{2\rho_0}(I-\tilde{I})\partial_{\tau} I -\frac{\gamma+1}{2\rho_0} \tilde{I}&\partial_{\tau}(I-\tilde{I}) -\frac{\nu}{2c^2\rho_0}\partial^2_{\tau}(I-\tilde{I})\\
&-\frac{c^2}{2}\partial_{\tau}^{-1} \Delta_y(I-\tilde{I})=\varepsilon\frac{\rho_0}{2c^2}R_{Kuz-KZK}.
\end{align*}
Denoting $\Omega_1=\mathbb{T}_{\tau}\times \mathbb{R}^{n-1}$, we multiply this equation by $(I-\tilde{I})$, integrate over $\mathbb{T}_{\tau}\times \mathbb{R}^{n-1}$ and perform a standard integration by parts, which gives
\begin{align*}
&\frac{c}{2}\frac{d}{dz}\Vert I-\tilde{I}\Vert_{L^2(\Omega_1)}^2-\frac{\gamma+1}{2\rho_0}\int_{\Omega_1}\partial_{\tau} I( I-\tilde{I})^2d\tau dy\\
&-\frac{\gamma+1}{2\rho_0} \int_{\Omega_1} \tilde{I} (I-\tilde{I}) \partial_{\tau}(I-\tilde{I})d\tau dy\\
&+\frac{\nu}{2c^2\rho_0}\int_{\Omega_1} (\partial_{\tau}(I-\tilde{I}))^2d\tau dy=\varepsilon\frac{\rho_0}{2c^2}\int_{\Omega_1} R_{Kuz-KZK} (I-\tilde{I})d\tau dy.
\end{align*}
Let us notice that
\begin{align*}
&\int_{\Omega_1} \tilde{I} (I-\tilde{I}) \partial_{\tau}(I-\tilde{I})d\tau dy= \int_{\Omega_1} [(\tilde{I}-I)+I)]\frac{1}{2}\partial_{\tau}(I-\tilde{I})^2 d\tau dy=\\
&=-\frac{1}{2}\int_{\Omega_1} \partial_{\tau}I (I-\tilde{I})^2 d\tau dy.
\end{align*}
By~(\ref{EqUTh21}) with $s> \max(\frac{n}{2},2)$, $u$ is sufficiently regular to ensure
\begin{equation}\label{remboundkuzkzk}
R_{Kuz-KZK}\in C(\mathbb{R}_+;L^2(\mathbb{T}_{\tau}\times \mathbb{R}^{n-1})).
\end{equation}
This comes from the fact that in system~(\ref{KZKwithre}) the ``worst'' term,
asking the most   regularity of $\Phi$, inside the remainder
$R_{Kuz-KZK}$ (see Eq.~(\ref{remkuzkzk})) is $\partial_{\tau}\partial^2_z\Phi$
with $\tilde{I}$ given by
Eq.~(\ref{tildI}). As
$\partial_t^3 u\in L^2(\mathbb{T}_t;H^{s-2}(\Omega)),$
we need to take $s> \max(\frac{n}{2},2)$ to have $\partial_{\tau}\partial^2_z\Phi$ in $L^{\infty}(\mathbb{R}_+;L^2(\mathbb{T}_{\tau}\times \mathbb{R}^{n-1}))$.
Therefore, it holds
$$\left\vert \int_{\Omega_1} R_{Kuz-KZK} (I-\tilde{I})d\tau dy\right\vert\leq \Vert R_{Kuz-KZK}\Vert_{L^2(\Omega_1)} \Vert I-\tilde{I}\Vert_{L^2(\Omega_1)}\leq C \Vert I-\tilde{I}\Vert_{L^2(\Omega_1)}$$
with a constant $C>0$ independent of $z$ thanks to~(\ref{remboundkuzkzk}). It leads to the estimate
$$\frac{1}{2}\frac{d}{dz}\Vert I-\tilde{I}\Vert_{L^2(\Omega_1)}^2\leq K \sup_{(\tau,y)\in\Omega_1}\vert \partial_{\tau}I(\tau,z,y)\vert \;\;\Vert I-\tilde{I}\Vert_{L^2(\Omega_1)}^2+C\varepsilon  \Vert I-\tilde{I}\Vert_{L^2(\Omega_1)},$$
in which, due to the regularity of $I$ for $s$ and $I_0$ (see~\cite{Roz2})
the term $$\sup_{(\tau,y)\in\Omega_1}\vert \partial_{\tau}I(\tau,z,y)\vert$$ is bounded by a constant $C>0$ independent of $z$. Consequently, we have the desired estimate and the other results follow from Gronwall's Lemma.
\end{proof}
\begin{remark}
The regularity $I_0\in H^{s+\frac{3}{2}}(\mathbb{T}_{\tau}\times\mathbb{R}^{n-1})$ for $s>\max(\frac{n}{2},2)$, imposed in Theorem~\ref{AproxKuzKZK}, is the minimal regularity to ensure~(\ref{remboundkuzkzk}).
\end{remark}

\subsubsection{Approximation problem for the Kuznetsov equation with initial-boundary conditions.}\label{sssKKZKperbcIn}
Let  the function $I_0(t,y)=I_0(t,\sqrt{\varepsilon }x')$ be
$L$-periodic on $t$  and  such that
$$I_0 \in H^s(\mathbb{T}_{t}\times \mathbb{R}^{n-1})$$
for $s\geq \left[ \frac{n+1}{2}\right]$ and $\int_{0}^L I_0(s,y) ds=0$. Hence~\cite{Roz2}, there is a unique solution $I(\tau,z,y)$ of the Cauchy problem~(\ref{NPEcau2}) for the KZK equation satisfying~(\ref{regsolKZK}).
We define $\overline{u}$ and $g$ as in Eqs.~(\ref{potentialkzk}) and~(\ref{boundcondkuz}) respectively.
Thus, for $R_{Kuz-KZK}$ defined in Eq.~(\ref{remkuzkzk}), $\overline{u}$ is the solution of the following system
\begin{equation}\label{aproxkuzkzkrem}
\left\lbrace
\begin{array}{l}
\partial^2_t \overline{u} -c^2 \Delta \overline{u}-\varepsilon\partial_t\left( (\nabla \overline{u})^2+\frac{\gamma-1}{2c^2}(\partial_t \overline{u})^2+\frac{\nu}{\rho_0}\Delta \overline{u}\right)=\varepsilon^2 R_{Kuz-KZK}\;\;\;\hbox{ in }\;\mathbb{T}_t\times \Omega,\\
\overline{u}=g\;\;\;\hbox{ on }\; \mathbb{T}_t\times\partial\Omega.
\end{array}\right.
\end{equation}

 In the same time let us consider for a sufficiently large $T>0$ the solution $u$ (see Theorem~\ref{ThMainWPnuPGlobHalf} for its global existence and uniqueness) of the Dirichlet boundary-value problem for the Kuznetsov equation
\begin{equation}\label{kuzhalfinitbound}
 \left\lbrace
 \begin{array}{l}
 u_{tt}-c^2\Delta u-\nu\varepsilon \Delta u_t=\alpha\varepsilon u_t u_{tt}+\beta \varepsilon \nabla u \nabla u_t\;\;\;\hbox{ in }\;[0,+\infty[\times \Omega,\\
 u=g\;\;\;\hbox{ on }\; [0,\infty[\times\partial\Omega,\\u(0)=u_0, \;\;\;u_t(0)=u_1\;\;\;\hbox{ in }\; \Omega,
\end{array}
\right.
\end{equation}
taking $u_0:=\overline{u}(0)$ and $u_1:=\overline{u}_t(0)$ and considering the time periodic function $g$ defined by Eq.~(\ref{boundcondkuz}) as a function on $[0,T]$.

To  compare $u$ and $\overline{u}$, we obtain the following stability result:
\begin{theorem}\label{approxKuzKZKbis}
Let $T$, $\nu>0$, $n\geq 2$, $s\in \R^+$,
$\Omega=\mathbb{R}^+\times\mathbb{R}^{n-1}$ and $I_0\in
H^{s}(\mathbb{T}_{t}\times\mathbb{R}^{n-1})$.

Then,  the following statements are
valid.
\begin{enumerate}
\item If $s\geq 6$ for $n=2$ and $3$, or else
$\left[\frac{s}{2}\right]>\frac{n}{2}+1$,   there exists  a constant
$C_0>0$ such that $\Vert
I_0\Vert_{H^{s}}< C_0$
implies the global well-posedness of the Cauchy problem
for the KZK equation with the following regularity:
$$\hbox{ for } 0\leq k\leq \left[\frac{s}{2}\right] \quad I\in C^k(\lbrace z >0\rbrace;H^{s-2k}(\mathbb{T}_{\tau}\times \mathbb{R}^{n-1})).$$
Moreover it implies the well-posedness of~(\ref{aproxkuzkzkrem}) with
$$
\overline{u}\in  C^k(\lbrace z >0\rbrace;H^{s-2k}(\mathbb{T}_{\tau}\times \mathbb{R}^{n-1})),\;
\partial_t\overline{u}\in  C^k(\lbrace z >0\rbrace;H^{s-2k}(\mathbb{T}_{\tau}\times \mathbb{R}^{n-1})),$$
or again
\begin{equation}\label{regubarkuzkzk}
\overline{u}\in H^2(\mathbb{T}_t,H^{\left[\frac{s}{2} \right]-1}(\Omega))\cap H^1(\mathbb{T}_t,H^{\left[\frac{s}{2} \right]}(\Omega)) .
\end{equation}
The  imposed regularity of $I_0$
(see Table~\ref{TABLE2}) is minimal to ensure  that $R_{Kuz-KZK}$ (see
Eq.~(\ref{remkuzkzk}) for the definition) is in
$C([0,+\infty[;L^2(\mathbb{R}_+\times\mathbb{R}^{n-1}))$.
 \item If $\left[\frac{s}{2}\right]>\frac{n}{2}+2$, taking the same initial data for the exact boundary-value problem for the Kuznetsov equation~(\ref{kuzhalfinitbound}) as for $\overline{u}$, $i.e.$   \begin{align*}
   &u(0)=\overline{u}(0)=\frac{c^2}{\rho_0}\partial_{\tau}^{-1}I(-\frac{x_1}{c},\varepsilon x_1,\sqrt{\varepsilon}x')\in H^{\left[\frac{s}{2} \right]}(\Omega),    \\
   &u_t(0)=\overline{u}_t(0)=\frac{c^2}{\rho_0}\partial_{\tau}I(-\frac{x_1}{c},\varepsilon x_1,\sqrt{\varepsilon}x')\in H^{\left[\frac{s}{2} \right]-1}(\Omega),
   \end{align*}
  there exists $ C_0>0$  such that $\Vert
I_0\Vert_{H^{s}}< C_0$ implies the
well-posedness of the exact Kuznetsov equation~(\ref{kuzhalfinitbound})
 supplemented with the Dirichlet boundary condition
 \begin{align*}
 g= \frac{c^2}{\rho_0}\partial_{\tau}^{-1}I_0\in H^s(\mathbb{T}_t\times\mathbb{R}^{n-1})\subset H^{7/4}(]0,T[&;H^{\left[\frac{s}{2} \right]-2}(\partial\Omega))\\
& \cap H^1(]0,T[;H^{\left[\frac{s}{2} \right]-2+3/2}(\partial\Omega))
 \end{align*}
  ensuring the regularity
\begin{equation}\label{regukuzkzk}
u\in H^2(]0,T[,H^{\left[\frac{s}{2} \right]-1}(\Omega))\cap H^1(]0,T[,H^{\left[\frac{s}{2} \right]}(\Omega)) .
\end{equation}
Moreover, there exist constants $K$, $C$, $C_1$, $C_2>0$, all independent of $\varepsilon$, such that for all $ t\leq\frac{C}{\varepsilon}$
\begin{equation}\label{estimKuzKZKex}
 \sqrt{\Vert (u -\overline{u})_t(t)\Vert_{L^2(\Omega)}^2+ \Vert \nabla (u-\overline{u})(t)\Vert_{L^2(\Omega)}^2}\leq C_1\eps^2t e^{C_2\eps t}\leq K\varepsilon.
\end{equation}

 \item In addition, let $u$ be a solution of the Dirichlet boundary-value problem~(\ref{kuzhalfinitbound}) for the Kuznetsov equation, with $g$ defined by Eq.~(\ref{boundcondkuz}) and  with initial data
$u_0\in H^{m+2}(\Omega)$, $u_1\in H^{m+1}(\Omega)$ for $m>\frac{n}{2}$ such that
\begin{equation}\label{EqSmallInD}
 \Vert (u -\overline{u})_t(0)\Vert_{L^2(\Omega)}^2+ \Vert \nabla (u-\overline{u})(0)\Vert_{L^2(\Omega)}^2\leq \delta^2 \leq \varepsilon^2.
\end{equation}
Then there exist constants $K$, $C$, $C_1$, $C_2>0$, all independent of $\varepsilon$, such that for all $ t\leq\frac{C}{\varepsilon}$
\begin{equation}\label{estimKuzKZK}
\sqrt{\Vert (u -\overline{u})_t(t)\Vert_{L^2(\Omega)}^2+ \Vert \nabla (u-\overline{u})(t)\Vert_{L^2(\Omega)}^2}\leq  C_1(\eps^2t+\delta^2)e^{C_2\eps t}\le K\eps.
\end{equation}
\end{enumerate}
\end{theorem}
\begin{proof}
Let $\overline{u}$ and $g$ be defined by Eqs.~(\ref{potentialkzk}) and~(\ref{boundcondkuz}) respectively by the solution $I$ of the Cauchy problem~(\ref{NPEcau2}) for the KZK equation with $I\vert_{z=0}=I_0\in H^s(\mathbb{T}_t\times\mathbb{R}^{n-1})$ and $s\geq 6$ for $n=2$ and $3$, or else $\left[\frac{s}{2}\right]>\frac{n}{2}+1$. In this case, $\overline{u}$ is the global solution of the approximated Kuznetsov system~(\ref{aproxkuzkzkrem}), what is a direct consequence of Theorem~1.2 in Ref.~\cite{Roz2}.
If $I_0\in H^s(\mathbb{T}_t\times \mathbb{R}^{n-1})$ with the chosen $s$, then
 $I\in C(\lbrace z >0\rbrace;H^{s}(\mathbb{T}_{\tau}\times
\mathbb{R}^{n-1}))$. But knowing, thanks to estimate~(\ref{EqEstimateAdrien}),
that $\Delta_y^k I_0\in H^{s-2k}(\mathbb{T}_t\times
 \R^{n-1})$ implies also $\del_{t}^{-k}\Delta_y^k I_0\in
H^{s-2k}(\mathbb{T}_t\times
 \R^{n-1})$ for $1\leq k\leq \left[\frac{s}{2}\right]$, the condition
in~\cite[Thm.~1.2, Point~4]{Roz2} is verified and thus we have the
following regularity  of $I$ on $z$:
for $0\leq k \leq \left[\frac{s}{2} \right]$
$$I(\tau,z,y)\in C^k(\lbrace z >0\rbrace;H^{s-2k}(\mathbb{T}_{\tau}\times \mathbb{R}^{n-1})).$$
 As $\overline{u}$ is defined by~(\ref{potentialkzk}),  we deduce (using as previously the notation $\Omega_1=\mathbb{T}_{\tau}\times \mathbb{R}^{n-1}$)
\begin{align*}
 \overline{u}(\tau,z,y)\hbox{ and }\partial_\tau\overline{u}(\tau,z,y)\in & C^k(\lbrace z >0\rbrace;H^{s-2k}(\Omega_1)),\hbox{ if }0\leq k\leq \left[\frac{s}{2} \right] ,\\
\partial_\tau^2\overline{u}(\tau,z,y)\in & C^k(\lbrace z >0\rbrace;H^{s-1-2k}(\Omega_1)),\hbox{ if }0\leq k\leq \left[\frac{s}{2} \right]-1,
\end{align*}
but we can also say~\cite{Ito,Roz2}, thanks to the exponential decay of the
solution of the KZK equation on $z$,
that
\begin{align*}
\overline{u}(\tau,z,y)\hbox{ and }\partial_\tau\overline{u}(\tau,z,y)\in &  H^k(\lbrace z >0\rbrace;H^{s-2k}(\Omega_1)),\\
\partial_\tau^2\overline{u}(\tau,z,y)\in & H^k(\lbrace z >0\rbrace;H^{s-1-2k}(\Omega_1)).
\end{align*}
This implies  for the chosen $s$ that
\begin{align*}
&\overline{u}(t,x_1,x')\hbox{ and } \partial_t\overline{u}(t,x_1,x')\in  L^2(\mathbb{T}_{t};H^{\left[\frac{s}{2} \right]}(\Omega))\cap H^2(\mathbb{T}_{t};H^{\left[\frac{s}{2} \right]-1}(\Omega)),\\
&\partial_t^2\overline{u}(t,x_1,x')\in  L^2(\mathbb{T}_{t};H^{\left[\frac{s}{2} \right]-1}(\Omega))\cap  H^2(\mathbb{T}_{t};H^{\left[\frac{s}{2} \right]-2}(\Omega)).
\end{align*}
Therefore
\begin{align*}
\overline{u}(t,x_1,x')\in & C^1([0,+\infty[;H^{\left[\frac{s}{2} \right]-1}(\Omega),\\
\partial_t^2\overline{u}(t,x_1,x')\in & C([0,+\infty[;H^{\left[\frac{s}{2} \right]-2}(\Omega).
\end{align*}
For the chosen $s$ these regularities of  $\overline{u}(t,x_1,x')$ give us
regularity~(\ref{regubarkuzkzk}) and allow to have all left-hand terms in the
approximated Kuznetsov system~(\ref{aproxkuzkzkrem}) of the desired regularity,
$i.e$  $C([0,+\infty[;L^2(\Omega))$.
In addition for $\left[\frac{s}{2}\right]>\frac{n}{2}+2$
with the chosen $g$, $u_0=\overline{u}(0)$ and $u_1=\overline{u}_t(0)$ in the conditions of the theorem
 we have
 $$u_0\in H^{\left[\frac{s}{2} \right]}(\Omega),\hbox{ }u_1\in H^{\left[\frac{s}{2} \right]-1}(\Omega) $$
 with
 $$g\in H^s(\mathbb{T}_t\times\mathbb{R}^{n-1}) \hbox{ and }\partial_t g \in H^s(\mathbb{T}_t\times\mathbb{R}^{n-1}).$$
 This implies
 $$g\in  H^{7/4}(]0,T[;H^{\left[\frac{s}{2} \right]-2}(\partial\Omega))\cap H^1(]0,T[;H^{\left[\frac{s}{2} \right]-2+3/2}(\partial\Omega))$$
 with $\left[\frac{s}{2} \right]-2 > \frac{n}{2}$,
 as required by Theorem~\ref{ThMainWPnuPGlobHalf} to have the  well-posedness of
 the solution of the Kuznetsov equation $u$ during the time  $t\in [0,T]$
associated   with system~(\ref{kuzhalfinitbound}).
This completes the well-posedness results and we deduce that $u$ have the desired regularity~(\ref{regukuzkzk}), announced in the theorem.
Moreover, we have $R_{Kuz-KZK}$ in $C([0,+\infty[,L^2(\Omega)).$

Let us now prove~(\ref{estimKuzKZK}) from point~$3$ as it directly implies
 estimate~(\ref{estimKuzKZKex}) from point~$2$.
We subtract the Kuznetsov equation from the approximated Kuznetsov equation (see system~(\ref{aproxkuzkzkrem})), multiply by $(u-\overline{u})_t$ and integrate over $\Omega$ to obtain, as in Ref.~\cite{Perso}, the following stability estimate:
\begin{align*}
\frac{1}{2}\frac{d}{dt}\Big(\int_{\Omega}A(t,x)\; (u-\overline{u})_t^2+ & c^2 (\nabla(u-\overline{u}))^2dx\Big) \\[2mm]
\leq C \varepsilon &\sup(\Vert u_{tt}\Vert_{L^{\infty}(\Omega)};\Vert \Delta u\Vert_{L^{\infty}(\Omega)};\Vert \nabla \overline{u}_t\Vert_{L^{\infty}(\Omega)})\\[2mm]
 & \cdot\left(\Vert (u -\overline{u})_t\Vert_{L^2(\Omega)}^2+ \Vert \nabla (u-\overline{u})\Vert_{L^2(\Omega)}^2\right)\\[2mm]
 &+\varepsilon^2\int_{\Omega}R_{Kuz-KZK} (u-\overline{u})_tdx,
\end{align*}
where $\frac{1}{2}\leq A(t,x)\leq \frac{3}{2}$ for $0\leq t\leq T$ and $x\in \Omega$. By regularity of the solutions, $\sup(\Vert u_{tt}\Vert_{L^{\infty}(\Omega)};\Vert \Delta u\Vert_{L^{\infty}(\Omega)};\Vert \nabla \overline{u}_t\Vert_{L^{\infty}(\Omega)})$ is bounded in time on $[0,T]$.
Moreover, we have $\Vert R_{Kuz-KZK}(t)\Vert_{L^2(\Omega)}$ bounded for $t\in[0,T]$ by the regularity of $\overline{u}$, where $R_{Kuz-KZK}$ is defined in Eq.~(\ref{remkuzkzk}). Then after integration on $[0,t]$, we can write
\begin{align*}
\Vert (u -\overline{u})_t(t)\Vert_{L^2(\Omega)}^2+ &\Vert \nabla (u-\overline{u})(t)\Vert_{L^2(\Omega)}^2\\[2mm]
\leq & 3( \Vert (u -\overline{u})_t(0)\Vert_{L^2(\Omega)}^2+ \Vert \nabla (u-\overline{u})(0)\Vert_{L^2(\Omega)}^2)\\[2mm]
& C_1 \varepsilon \int_0^t \Vert (u -\overline{u})_t(s)\Vert_{L^2(\Omega)}^2+ \Vert \nabla (u-\overline{u})(s)\Vert_{L^2(\Omega)}^2 ds\\[2mm]
&+C_2 \varepsilon^2 \int_0^t\sqrt{\Vert (u -\overline{u})_t(s)\Vert_{L^2(\Omega)}^2+ \Vert \nabla (u-\overline{u})(s)\Vert_{L^2(\Omega)}^2 }ds.
\end{align*}
Thanks to~(\ref{EqSmallInD})
we finally find by the Gronwall Lemma  that for $t\leq \frac{C}{\varepsilon}$
 estimate~(\ref{estimKuzKZK}) holds true, thereby concluding the
proof.
\end{proof}
\begin{remark}\label{RemInvCKZK}
Let us discuss the corresponding approximation results in the inviscid case.
 We have two approximation results:
 \begin{enumerate}
  \item between the solutions  $\overline{\mathbf{U}}_{KZK}$ of the KZK equation
and $\mathbf{U}_{Euler}$ of the Euler system~\cite{Dekkers,Roz3}
(see~\cite[Thm~6.8]{Dekkers} for the definitions of
$\mathbf{U}_{Euler}$ and $\overline{\mathbf{U}}_{KZK}$)    in a cone
 $$C(T)=\lbrace 0<t<T \vert \;T<\frac{T_0}{\varepsilon}\rbrace\times Q_{\varepsilon}(t)$$
 with $$Q_{\varepsilon}(s)=\lbrace x=(x_1,x'):\vert x_1\vert\leq \frac{R}{\varepsilon}-Ms, \;M\geq c, x'\in \mathbb{R}^{n-1}\rbrace$$ and with
 $$\Vert \nabla \mathbf{U}_{Euler} \Vert_{L^{\infty}([0,\frac{T_0}{\varepsilon}[;H^{s-1}(Q_{\varepsilon}))}<\varepsilon C \hbox{ for }s>\left[\frac{n}{2}\right]+1;$$
 \item between the solutions $\mathbf{U}_{Euler}$ of the Euler system and
$\overline{\mathbf{U}}_{Kuzn}$ of the Kuznetsov equation~\cite{Dekkers}
(see~\cite[Thm.~6.6]{Dekkers} for the definitions of
$\mathbf{U}_{Euler}$ and $\overline{\mathbf{U}}_{Kuzn}$) in
$[0,\frac{T_0}{\eps}[\times\R^n$ containing $C(T)$.
 \end{enumerate}
 Consequently,  we obtain the approximation result between the solutions $\overline{\mathbf{U}}_{KZK}$ of the KZK equation and the solutions $\overline{\mathbf{U}}_{Kuzn}$ of the Kuznetsov equation in $C(T)$ by the triangular inequality:
 $$\|\overline{\mathbf{U}}_{Kuzn}-\overline{\mathbf{U}}_{KZK}\|_{L^2(Q_\eps(t))}^2\leq K(\eps^3t+\delta^2)  e^{K\varepsilon t}\leq 9\varepsilon^2,$$
 as soon as
 $\Vert(\overline{\mathbf{U}}_{Kuzn}-\overline{\mathbf{U}}_{KZK})(0)\Vert_{L^2(Q_{\varepsilon}(0))}\leq \delta<\eps$.
 The initial data are constructed on the initial data $I_0$ for the KZK equation. More precisely we take $I_0\in H^s(\mathbb{T}_{\tau}\times \mathbb{R}^{n-1})$ for $s>\max\lbrace 10, \left[\frac{n}{2}\right]+1\rbrace,$ which ensures in the case of the same initial data
 $$\overline{\mathbf{U}}_{Kuzn}(0)=\overline{\mathbf{U}}_{KZK}(0)=\mathbf{U}_{Euler}(0)$$
 the existence with necessary regularity of all solutions: of the KZK equation, of the Euler system and of the Kuznetsov equation. Otherwise, to ensure the boundness and the minimal regularity  $C([0,\frac{T_0}{\varepsilon}[;L^2(Q_\varepsilon))$ of the remainder terms it sufficient to impose $s\ge 6$.
\end{remark}

\section{Approximation of the solutions of the Kuznetsov equation with the solutions of the NPE equation.}\label{secKuzNPE}
Now let us go back to the NPE equation~(\ref{NPE2}) and consider its
\textit{ansatz} (see~\cite{Dekkers} for the derivation of the NPE equation from
the isentropic Navier-Stokes system or  the Euler system). In
contrast with Eq.~(\ref{paraxpot}) for the KZK equation, this time  the
velocity potential is given~\cite{Roz1} by
\begin{equation}\label{paraxpotnpe}
u(x,t)=\Psi(\varepsilon t, x_1- c t,\sqrt{\epsilon}x')=\Psi(\tau,z,y).
\end{equation}
Thus we directly obtain from the Kuznetsov equation~(\ref{KuzEq}) with the paraxial change of variable
\begin{equation}\label{chvarnpe}
\tau= \varepsilon t,\;\;\;z=x_1- c t,\;\;\; y=\sqrt{\varepsilon}x',\end{equation}
  that
\begin{align*}
&\partial_t^2 u-c^2\Delta u- \varepsilon \partial_t\left((\nabla u)^2+\frac{\gamma-1}{2c^2}(\partial_t u)^2+\frac{\nu}{\rho_0}\Delta u\right)\\
&=\varepsilon \left(-2c \partial^2_{\tau z} \Psi-c^2\Delta_y\Psi+\frac{\nu}{\rho_0} c\partial^3_z \Psi+\frac{\gamma+1}{2}c \partial_z(\partial_z \Psi)^2\right)+\varepsilon^2 R_{Kuz-NPE}
\end{align*}
with
\begin{align}
\varepsilon^2 R_{Kuz-NPE}=&\varepsilon^2 \big( \partial^2_{\tau}\Psi-\frac{\nu}{\rho_0} \partial^2_z\partial_{\tau}\Psi+\frac{\nu}{\rho_0}c\Delta_y\partial_z\Psi-(\gamma-1)\partial_{\tau}\Psi \;\partial^2_z\Psi\label{approxeqKuzNPE}\\
&\;\;\;\;\;-2(\gamma-1)\partial_z\Psi\; \partial^2_{\tau z}\Psi-2 \partial_z\Psi \;\partial^2_{\tau z}\Psi+2c \nabla_y\Psi\;\nabla_y\partial_z\Psi \big)\nonumber\\
&+\varepsilon^3 \big( -\frac{\nu}{\rho_0}\Delta_y \partial_{\tau}\Psi+2\frac{\gamma-1}{c}\partial_{\tau}\Psi\;\partial^2_{\tau z}\Psi+\frac{\gamma-1}{c}\partial_z \Psi\; \partial^2_{\tau}\Psi\nonumber\\
&\;\;\;\;\;\;\;\;\;\; -2\nabla_y\Psi\;\nabla_y\partial_{\tau}\Psi\big)
+\varepsilon^4 (-\frac{\gamma-1}{c^2}\partial_{\tau}\Psi \partial^2_{\tau}\Psi).\nonumber
\end{align}
We obtain the NPE equation satisfying by $\partial_z\Psi$ modulo a multiplicative constant:
$$
\partial^2_{\tau z}\Psi-\frac{\gamma+1}{4}\partial_z(\partial_z \Psi)^2-\frac{\nu}{2\rho_0}\partial^3_z \Psi+\frac{c}{2}\Delta_y \Psi=0.
$$
In the sequel we  work with $\xi$ defined by
\begin{align}
\xi(\tau,z,y)=&-\frac{\rho_0}{c}\partial_z\Psi,\label{P1NPE}
\end{align}
 which solves the Cauchy problem for the NPE equation
\begin{equation}\label{npecau}
\left\lbrace
\begin{array}{c}
\partial^2_{\tau z} \xi+\frac{(\gamma+1)c}{4\rho_0}\partial_z^2[(\xi)^2]-\frac{\nu}{2\rho_0}\partial^3_z  \xi+\frac{c}{2}\Delta_y  \xi=0\hbox{ on }\mathbb{R}_+\times\mathbb{T}_{z}\times\mathbb{R}^{n-1},\\
\xi(0,z,y)=\xi_0(z,y)\hbox{ on }\mathbb{T}_{z}\times\mathbb{R}^{n-1},
\end{array}\right.
\end{equation}
in the class of $L-$periodic functions with respect to the variable $z$ and with mean value zero along $z$.
The introduction of the operator $\partial_z^{-1}$ defined similarly to $\partial_{\tau}^{-1}$ in Eq.~(\ref{invdtau})
allows us to consider instead of Eq.~(\ref{NPE2}) the following equivalent equation
$$\partial_{\tau }
\xi+\frac{(\gamma+1)c}{4\rho_0}\partial_z[(\xi)^2]-\frac{\nu}{2\rho_0}
\partial^2_z  \xi+\frac{c}{2}\partial_{z}^{-1}\Delta_y
\xi=0\hbox{ on }\mathbb{R}_+\times\mathbb{T}_{z}\times\mathbb{R}^{n-1} .$$
This time, in comparison with the KZK equation, we use the bijection between this two models (see~\cite{Dekkers}). We also update our notation for
$\Omega_1=\mathbb{T}_z\times \mathbb{R}^{n-1}_y$  and take $s>\frac{n}{2}+1$. Suppose that $$\xi_0\in H^{s+2}(\mathbb{T}_z\times \mathbb{R}^{n-1}_y) \quad \hbox{and} \quad\int_{\mathbb{T}_z}\xi_0(z,y)\;dz=0.$$
Consequently there exists a constant $r>0$ such that if $\Vert
\xi_0\Vert_{H^{s+2}(\mathbb{T}_z\times \mathbb{R}^{n-1}_y)}<r$, then,
by~\cite[Thms.~1.1,~1.2]{Roz2}, there exists
a unique solution
$$\xi\in C([0,\infty[;H^{s+2}(\mathbb{T}_z\times \mathbb{R}^{n-1}_y))$$ of
 the NPE Cauchy problem~(\ref{npecau})
satisfying
$$\int_{\mathbb{T}_z} \xi(\tau,z,y) \;dz=0 \;\;\hbox{ for any }\;\tau\geq 0,\; y\in\mathbb{R}^{n-1}.$$
We define $\partial_{x_1} \overline{u}(t,x_1,x'):=-\frac{c}{\rho_0}\xi(\tau,z,y)$ with the change of variable~(\ref{chvarnpe})
 and
$$\overline{u}(t,x_1,x')=-\frac{c}{\rho_0} \partial_z^{-1} \xi(\tau,z,y)=\left(-\frac{c}{\rho_0}\right)\left(\int_0^z \xi(\tau,s,y)ds+\int_0^L \frac{s}{L} \xi(\tau,s,y)ds\right).$$
We take $u_1(x_1,x'):= \partial_t\overline{u}(0,x_1,x') $ and $u_0(x_1,x'):=-\frac{c}{\rho_0}\partial_z^{-1}\xi_0(z,y)$, which implies
$$u_0\in H^{s+2}(\mathbb{T}_{x_1}\times \mathbb{R}^{n-1}_{x'})\hbox{ and }u_1\in H^{s}(\mathbb{T}_{x_1}\times \mathbb{R}^{n-1}_{x'}).$$
Thus for these initial data there exists
$$\overline{u}\in C([0,\infty[;H^{s+1}(\mathbb{T}_{x_1}\times \mathbb{R}^{n-1}_{x'}))\cap C^1([0,\infty[;H^{s}(\mathbb{T}_{x_1}\times \mathbb{R}^{n-1}_{x'})),$$
the unique solution on $\mathbb{T}_{x_1}\times \mathbb{R}^{n-1}_{x'}$ of the approximated Kuznetsov system
   \begin{equation}\label{CauchyaproxKuzNPE}
\left\lbrace
\begin{array}{c}
\overline{u}_{tt}-c^2 \Delta \overline{u}-\nu \varepsilon \Delta \overline{u}_t-\alpha \varepsilon \overline{u}_t \overline{u}_{tt}-\beta \varepsilon \nabla \overline{u} \nabla \overline{u}_t=\varepsilon^2 R_{Kuz-NPE},\\
\overline{u}(0)=u_0\in H^{s+2}(\mathbb{T}_{x_1}\times \mathbb{R}^{n-1}_{x'}),\;\;\;\overline{u}_t(0)=u_1\in H^{s+1}(\mathbb{T}_{x_1}\times \mathbb{R}^{n-1}_{x'})
\end{array}\right.
\end{equation}
with $R_{Kuz-NPE}$ defined in Eq.~(\ref{approxeqKuzNPE}).
If we consider the Cauchy problem
\begin{equation}\label{CauProbKuz}
\left\lbrace
\begin{array}{l}
 \partial^2_t u -c^2 \Delta u=\varepsilon\partial_t\left( (\nabla u)^2+\frac{\gamma-1}{2c^2}(\partial_t u)^2+\frac{\nu}{\rho_0}\Delta u\right),\\
u(0)=u_0,\;\;u_t(0)=u_1,
\end{array}
\right.
\end{equation}
 for the Kuznetsov equation on $\mathbb{T}_{x_1}\times \mathbb{R}^{n-1}_{x'}$ with $u_0$ and $u_1$ derived from $\xi_0$, we have
$$\Vert u_0\Vert_{H^{s+2}(\mathbb{T}_{x_1}\times \mathbb{R}^{n-1}_{x'})}+\Vert u_1\Vert_{H^{s}(\mathbb{T}_{x_1}\times \mathbb{R}^{n-1}_{x'})}\leq C \Vert \xi_0\Vert_{H^{s+2}(\mathbb{T}_z\times \mathbb{R}^{n-1}_y)}.$$
Hence, if $\Vert \xi_0\Vert_{H^{s+2}(\mathbb{T}_z\times \mathbb{R}^{n-1}_y)}$
 is
small enough~\cite{Perso}, we have a unique bounded in time solution
$$u\in C([0,\infty[;H^{s+1}(\Omega))\cap C^1([0,\infty[;H^{s}(\Omega))$$
 of the Cauchy problem for the Kuznetsov equation~(\ref{CauProbKuz}).
\begin{theorem}\label{approxKuzNPE}
For $\nu\ge 0$ let $u$ and $\overline{u}$ be  the defined above solutions  of the exact Cauchy problem~(\ref{CauProbKuz}) and  of the approximated Cauchy problem~(\ref{CauchyaproxKuzNPE}) for the Kuznetsov equation on $\Omega= \mathbb{T}_{x_1}\times \mathbb{R}^{n-1}_{x'}$  respectively.
 Then for $\nu>0$
there exist $K$, $C$, $C_1$,  $C_2>0$  such that for all
$t<\frac{C}{\varepsilon} $
 estimate~(\ref{estimKuzKZKex})   is valid  and in addition it
holds Point~3 of Theorem~\ref{approxKuzKZKbis}.

Moreover, if for $n\leq 3$,  and $\xi_0\in
H^s(\mathbb{T}_{x_1}\times \mathbb{R}^{n-1})$ with $s \geq 4$, then the
approximated solution
satisfies
\begin{align*}
&\overline{u}(t,x_1,x')\in C([0,+\infty[;H^4(\Omega)),\;
\partial_t\overline{u}(t,x_1,x')\in C([0,+\infty[;H^2(\Omega)),\\
&\partial_t^2\overline{u}(t,x_1,x')\in C([0,+\infty[;L^2(\Omega)).
\end{align*}

If for $n \geq 4$ $\xi_0\in H^s(\mathbb{T}_{x_1}\times \mathbb{R}^{n-1})$ with $s \geq \frac{n}{2}+2$, then the approximated solution
satisfies
\begin{align*}
&\overline{u}(t,x_1,x')\in C([0,+\infty[;H^s(\Omega)),\;
\partial_t\overline{u}(t,x_1,x')\in C([0,+\infty[;H^{s-2}(\Omega)),\\
&\partial_t^2\overline{u}(t,x_1,x')\in C([0,+\infty[;H^{s-4}(\Omega)).
\end{align*}
Under these conditions for $n\geq 1$
$$
R_{Kuz-NPE}\in  C([0,+\infty[;L^2(\mathbb{T}_{x_1}\times \mathbb{R}^{n-1})).
$$
For $\nu=0$ all previous results stay true on a finite time interval $[0,T]$.
\end{theorem}
\begin{proof}
For $\nu>0$ the global existence of $u$ and of $\overline{u}$ has already been shown.
The proof of the approximation estimate follows exactly the proof given for Theorem~\ref{approxKuzKZKbis} and thus is  omitted.
The case $\nu=0$ implies the same approximation result except that $u$ and $\overline{u}$ are only locally well posed on an interval $[0,T]$.

We can see for $n=2$ or $3$, using the previous arguments that the minimum regularity of the initial data (see Table~\ref{TABLE2}) to have the remainder terms
$$
R_{Kuz-NPE}\in  C([0,+\infty[;L^2(\mathbb{T}_{x_1}\times \mathbb{R}^{n-1}))
$$
corresponds to
$\xi_0\in H^s(\mathbb{T}_{x_1}\times \mathbb{R}^{n-1})$ with $s \geq 4$, since then for $0\leq k \leq 2$
$$\xi(\tau,z,y)\in C^k([0,+\infty[\rbrace;H^{s-2k}(\mathbb{T}_{z}\times \mathbb{R}^{n-2})),$$
which finally implies with formula $\overline{u}=-\frac{c}{\rho_0}\partial_z^{-1}\xi $ that
\begin{align*}
&\overline{u}(t,x_1,x')\in C([0,+\infty[;H^4(\Omega)),\;
\partial_t\overline{u}(t,x_1,x')\in C([0,+\infty[;H^2(\Omega)),\\
&\partial_t^2\overline{u}(t,x_1,x')\in C([0,+\infty[;L^2(\Omega)).
\end{align*}
In the same way for $n\geq 4$ we find the minimal regularity for $\xi_0\in H^s(\Omega)$ with $s >\frac{n}{2}+2$ as it implies
\begin{align*}
&\overline{u}(t,x_1,x')\in C([0,+\infty[;H^{s}(\Omega)),\;
\partial_t\overline{u}(t,x_1,x')\in C([0,+\infty[;H^{s-2}(\Omega)),\\
&\partial_t^2\overline{u}(t,x_1,x')\in C([0,+\infty[;H^{s-4}(\Omega)).
\end{align*}
The optimality of the  previously chosen $s$ also comes from the fact that in Eq.~(\ref{approxeqKuzNPE}) the least regular term in $R_{Kuz-NPE}$ is $\partial_{\tau} \Psi \partial_{\tau}^2 \Psi $ presenting for both viscous and inviscid cases.
\end{proof}

\section{The Kuznetsov equation and the Westervelt equation}\label{SecWest}
\subsection{Derivation of the Westervelt equation from the Kuznetsov equation.}\label{subSecWestD}
Let $u$ be a solution of  the Kuznetsov equation~(\ref{KuzEq}).
Similarly as in Ref.~\cite{Aanonsen} we set
\begin{equation}\label{ansatzKuzWes}
\overline{\Pi}=u+\frac{1}{2 c^2}\varepsilon \partial_t[u^2]
\end{equation}
and obtain
$$\partial_t^2\overline{\Pi}-c^2\Delta \overline{\Pi}=\varepsilon \partial_t\left(\frac{\nu}{\rho_0}\Delta u+\frac{\gamma+1}{2c^2} (\partial_t u)^2+\frac{1}{c^2}u(\partial_t^2-c^2\Delta u)\right).$$
By definition~(\ref{ansatzKuzWes}) of $\overline{\Pi}$ we have
\begin{equation}\label{EqWestApp}
 \partial_t^2\overline{\Pi}-c^2\Delta \overline{\Pi}=\varepsilon \partial_t\left(\frac{\nu}{\rho_0}\Delta \overline{\Pi}+\frac{\gamma+1}{2c^2} (\partial_t \overline{\Pi})^2\right)+\varepsilon^2 R_{Kuz-Wes},
\end{equation}
where
\begin{align}
\varepsilon^2 R_{Kuz-Wes}=&\varepsilon^2 \partial_t\left[ -\frac{1}{2c^2}\frac{\nu}{\rho_0}\Delta(u \partial_t u)-\frac{\gamma+1}{2c^4}\partial_t u\partial^2_t(u^2)\right.\nonumber\\
&\left.\;\;\;\; +\frac{1}{c^2}u\partial_t\left( (\nabla u)^2+\frac{\gamma-1}{2c^2}(\partial_t u)^2+\frac{\nu}{\rho_0}\Delta u\right)\right]\nonumber\\
&+\varepsilon^3 \partial_t \left[-\frac{\gamma+1}{8c^6}[\partial^2_t(u^2)]^2\right].\label{resWes}
\end{align}
We recognize the Westervelt equation~(\ref{West}) obtained up to remainder terms
of order~$\eps^2$.

\subsection{Approximation of the solutions of the Kuznetsov equation by the solutions of the Westervelt equation}\label{subSecWestA}
The well-posedness of the Westervelt equation  follows directly from~\cite{Perso}.
For a solution of the Cauchy problem~(\ref{CauProbKuz}) for the Kuznetsov equation $u$ we define as in Subsection~\ref{subSecWestD} $\overline{\Pi}$ by
 Eq.~(\ref{ansatzKuzWes}). Hence $\overline{\Pi}$ is the solution of the approximated Cauchy problem for the Westervelt equation~(\ref{EqWestApp}) with the initial data
 \begin{equation}\label{EqInDataWAp}
  \overline{\Pi}(0)=\Pi_0\hbox{, }\quad\partial_t\overline{\Pi}(0)=\Pi_1,
 \end{equation}
 defined by
\begin{align}
\Pi_0=&u_0+\frac{1}{c^2}\varepsilon u_0 u_1,\label{pi0}\\
\Pi_1=&u_1+\frac{1}{c^2}\varepsilon u_1^2+ \frac{1}{c^2}\varepsilon u_0 \partial_t^2u(0)\label{pi1}\\
=&u_1+\frac{1}{c^2}\varepsilon u_1^2+ \frac{1}{c^2}\varepsilon u_0\frac{1}{1-\frac{\gamma-1}{c^2}\varepsilon u_1}\left(c^2\Delta u_0 +\frac{\nu}{\rho_0}\varepsilon \Delta u_1+2\varepsilon \nabla u_0 \nabla u_1\right)\nonumber
\end{align}
with $u_0$ and $u_1$ initial data of the Cauchy problem~(\ref{CauProbKuz}) for the Kuznetsov equation.

For $s>\frac{n}{2}$ and $\nu>0$, if we take $u_0\in H^{s+3}(\mathbb{R}^n)$ and $u_1\in H^{s+3}(\mathbb{R}^3)$, we have $\Pi_0\in H^{s+3}(\mathbb{R}^n)\subset H^{s+2}(\mathbb{R}^n)$ and $\Pi_1 \in H^{s+1}(\mathbb{R}^n)$ with
$$\Vert \Pi_0\Vert_{H^{s+2}(\mathbb{R}^n)}+\Vert \Pi_1\Vert_{H^{s+1}(\mathbb{R}^n)}\leq C (\Vert u_0\Vert_{H^{s+3}(\mathbb{R}^n)}+\Vert u_1\Vert_{H^{s+3}(\mathbb{R}^n)} ).$$
In the inviscid case when $\nu=0$, for $s>\frac{n}{2}$  if we  still take $u_0\in H^{s+3}(\mathbb{R}^n)$, but $u_1\in H^{s+2}(\mathbb{R}^3)$, we have $\Pi_0\in  H^{s+2}(\mathbb{R}^n)$ and $\Pi_1 \in H^{s+1}(\mathbb{R}^n)$ with the estimate
$$\Vert \Pi_0\Vert_{H^{s+2}(\mathbb{R}^n)}+\Vert \Pi_1\Vert_{H^{s+1}(\mathbb{R}^n)}\leq C (\Vert u_0\Vert_{H^{s+3}(\mathbb{R}^n)}+\Vert u_1\Vert_{H^{s+2}(\mathbb{R}^n)} ).$$
Then, similarly to our previous work~\cite{Perso}, we obtain  the
following result.
\begin{theorem}\label{ThWPWesR3}
Let $n\geq 1$ and $s>\frac{n}{2}$.
\begin{enumerate}
\item If $\nu>0$, $u_0\in H^{s+3}(\mathbb{R}^n)$ and $u_1\in H^{s+3}(\mathbb{R}^n)$, then there exists a constant $k_2>0$ such that for
\begin{equation}\label{smallcondWPwes}
\Vert u_0\Vert_{H^{s+4}(\mathbb{R}^n)}+\Vert u_1\Vert_{H^{s+3}(\mathbb{R}^n)} <k_2,
\end{equation}
the exact Cauchy problem for the Westervelt equation
\begin{equation}\label{CauPbWes}
\left\lbrace
\begin{array}{l}
\partial_t^2\Pi-c^2\Delta \Pi=\varepsilon \partial_t\left(\frac{\nu}{\rho_0}\Delta \Pi+\frac{\gamma+1}{2c^2} (\partial_t \Pi)^2\right),\\
\Pi(0)=\Pi_0\hbox{, }\partial_t\Pi(0)=\Pi_1
\end{array}
\right.
\end{equation}
with $\Pi_0$ and $\Pi_1$ defined by Eqs.~(\ref{pi0}) and~(\ref{pi1}),
 has a unique global in time solution
\begin{align}
\Pi  & \in H^2(]0,+\infty[,H^s(\mathbb{R}^n))\cap H^1(]0,+\infty[,H^{s+2}(\mathbb{R}^n))\label{RegWes31}\\
\hbox{ and } &\hbox{if } s\geq 1  \nonumber\\
\Pi  &\in C([0,+\infty[,H^{s+2}(\mathbb{R}^n))\cap C^1([0,+\infty[,H^{s+1}(\mathbb{R}^n))\cap C^2([0,+\infty[,H^{s-1}(\mathbb{R}^n))\label{RegWes32}
\end{align}
Moreover, $\overline{\Pi}$, obtained from the solution $u$ of the Kuznetsov equation with Eq.~(\ref{ansatzKuzWes}), is the unique global in time solution of the approximated Cauchy problem~(\ref{EqWestApp}),~(\ref{EqInDataWAp}) with the same regularity as $\Pi$.
\item  Let $\nu = 0$, $u_0\in H^{s+3}(\mathbb{R}^n)$ and $u_1\in H^{s+2}(\mathbb{R}^n)$. Then there exists a constant $k_2>0$ such that if
\begin{equation}\label{smallcondWPwesnu0}
\Vert u_0\Vert_{H^{s+3}(\mathbb{R}^n)}+\Vert u_1\Vert_{H^{s+2}(\mathbb{R}^n)} <k_2,
\end{equation}
then the Cauchy problem~(\ref{CauPbWes}) for the Westervelt equation
with $\Pi_0$ and $\Pi_1$, defined by Eqs.~(\ref{pi0}) and~(\ref{pi1}),
 has a unique  solution on all finite time interval $[0,T]$
\begin{align}
\Pi  &\in C([0,T],H^{s+2}(\mathbb{R}^n))\cap C^1([0,T],H^{s+1}(\mathbb{R}^n))\cap C^2([0,T],H^{s}(\mathbb{R}^n)).\label{RegWesnu0}
\end{align}
Moreover, $\overline{\Pi}$, defined by Eq.~(\ref{ansatzKuzWes}), is the unique local in time solution of the approximated Cauchy problem~(\ref{EqWestApp}),~(\ref{EqInDataWAp}) with the same regularity as $\Pi$.
\end{enumerate}
\end{theorem}
For $\Pi$, solution of the Cauchy problem for the Westervelt equation~(\ref{CauPbWes}), we set $\overline{u}$ such that
\begin{equation}\label{EqPi-ubar}
 \Pi=\overline{u}+\frac{\varepsilon}{c^2}\overline{u}\partial_t\overline{u}
\end{equation}
 and we obtain
\begin{align*}
\partial_t^2 \overline{u}-c^2 \Delta \overline{u}-\varepsilon\frac{\nu}{\rho_0}\Delta &\partial_t\overline{u}-\varepsilon \frac{\gamma -1}{c^2}\partial_t  \overline{u} \partial^2_t\overline{u}-2  \varepsilon \nabla \overline{u}.\nabla\partial_t  \overline{u}\\
+&\varepsilon\left(\frac{1}{c^2}\partial_t  \overline{u} \partial^2_t\overline{u}-\partial_t  \overline{u}\Delta \overline{u} +\frac{1}{c^2} \overline{u} \partial_t^3 \overline{u} - \overline{u}\Delta\partial_t  \overline{u}\right)=\varepsilon^2 R_{1,Wes-Kuz}
\end{align*}
with
\begin{align*}
R_{1,Wes-Kuz}=\left[ \frac{\nu}{\rho_0 c^2}\right.& (2\partial_t  \overline{u} \Delta\partial_t \overline{u}+2(\nabla \partial_t\overline{u})^2+\partial_t^2 \overline{u}\Delta \overline{u}+\overline{u}\Delta\partial^2_t+2\nabla \overline{u}.\nabla \partial^2_t \overline{u})\\
 &\left.+\frac{\gamma +1}{c^4}((\partial_t\overline{u})^2+\overline{u}\partial^2_t \overline{u})\partial^2_t\overline{u}+\frac{\gamma +1}{c^4} (3\partial_t \overline{u} \partial^2_t \overline{u}+\overline{u}\partial^3_t \overline{u})\partial_t \overline{u}\right]\\
 +\varepsilon \frac{\gamma +1}{c^6}& ((\partial_t\overline{u})^2+\overline{u}\partial^2_t \overline{u})(3\partial_t \overline{u} \partial^2_t \overline{u}+\overline{u}\partial^3_t \overline{u}).
\end{align*}
And as $$\partial_t^2 \overline{u}-c^2 \Delta \overline{u}=O(\varepsilon), $$ by inserting this in the term $\left(\frac{1}{c^2}\partial_t  \overline{u} \partial^2_t\overline{u}-\partial_t  \overline{u}\Delta \overline{u} +\frac{1}{c^2} \overline{u} \partial_t^3 \overline{u} -\varepsilon \overline{u}\Delta\partial_t  \overline{u}\right)$
we obtain
\begin{equation}\label{eqWesKuzRem}
\partial_t^2 \overline{u}-c^2 \Delta \overline{u}-\varepsilon\frac{\nu}{\rho_0}\Delta \partial_t\overline{u}-\varepsilon \frac{\gamma -1}{c^2}\partial_t  \overline{u} \partial^2_t\overline{u}-2  \varepsilon \nabla \overline{u}.\nabla\partial_t  \overline{u} =\varepsilon^2 R_{Wes-Kuz}
\end{equation}
with
$$\varepsilon^2 R_{Wes-Kuz}= \varepsilon^2R_{1,Wes-Kuz}-\varepsilon\left(\frac{1}{c^2}\partial_t  \overline{u} \partial^2_t\overline{u}-\partial_t  \overline{u}\Delta \overline{u} +\frac{1}{c^2} \overline{u} \partial_t^3 \overline{u} - \overline{u}\Delta\partial_t  \overline{u}\right).$$

Now we can write the following result for the approximation of the Kuznetsov equation by the Westervelt equation.
\begin{theorem}\label{ApproxKuzWes}
Let $n\geq 2$,  $s>\frac{n}{2}$ with $s\geq 1$ and $\nu\geq 0$.

Let $\overline{u}_0\in H^{s+3}(\mathbb{R}^n)$ and $\overline{u}_1\in H^{s+3}(\mathbb{R}^n)$ if $\nu>0$ and let $\overline{u}_1\in H^{s+2}(\mathbb{R}^n)$ if $\nu=0$ be small enough in the sense of the existence of $\Pi$ the solution of the Cauchy problem for the Westervelt equation~(\ref{CauPbWes}) with $\Pi_0$ and $\Pi_1$ defined by Eqs.~(\ref{pi0}) and~(\ref{pi1}).
Let $\overline{u}$ be defined by~(\ref{EqPi-ubar}).

Consequently $\overline{u}$ is a solution of the approximated Kuznetsov equation~(\ref{eqWesKuzRem}) with $\overline{u}(0)=\overline{u}_0$, $\partial_t\overline{u}(0)=\overline{u}_1$.
If the initial data for  $u$, the solution of the Cauchy problem~(\ref{CauProbKuz}) for the Kuznetsov equation, and  for $\overline{u}$ satisfy~(\ref{EqSmallInD}),
 there exist $K$, $C$,
$C_1$, $C_2>0$, all independent of $\varepsilon$, such that for all
$t\leq\frac{C}{\varepsilon}$
it holds estimate~(\ref{estimKuzKZK}).
\end{theorem}
\begin{proof}
The existence of $u$ and $\overline{u}$ has already been shown
 in~\cite{Perso} and given in  Theorem~\ref{ThWPWesR3}.
The proof of the approximation estimate follows exactly the proof of Theorem~\ref{approxKuzKZKbis} and hence it is omitted.
The regularity on $\overline{u}_0$ and $\overline{u}_1$ (see expressions of $u_0$ and $u_1$ in Table~\ref{TABLE2}) is minimal to ensure  that $R_{Wes-Kuz}$ (see Eq.~(\ref{eqWesKuzRem})) is in $C([0,+\infty[;L^2(\mathbb{R}^n))$.
Indeed, with $\Pi_0$ and $\Pi_1$ defined by Eqs.~(\ref{pi0}) and~(\ref{pi1}) it is necessary to impose these regularities in order to have the well-posedness of $\Pi$ with the same regularity as in Theorem~\ref{ThWPWesR3}.
\end{proof}

 \appendix\label{annex}

 \section{Well posedness of the Kuznetsov equation in the half space.}\label{secWPresults}
We establish here the well posedness results for the Kuznetsov equation in the framework of the KZK-approximation considered in Subsection~\ref{secValKuzKZK}. For two type of approximations we need two different well posedness results.

\subsection{
Periodic boundary problem.}\label{WPKuzhalf1}
Let us consider the following periodic in time problem for the Kuznetsov equation in the half space $\Omega=\mathbb{R}_+\times\mathbb{R}^{n-1}$ with periodic in time Dirichlet boundary conditions given by~(\ref{kuzper}), 
where $g$ is an $L$-periodic in time and of mean value zero function.
To show the well-posedness of problem~(\ref{kuzper}) we study the maximal regularity of the associated linear operator and then use an equivalent to the fixed point theorem.
Using~\cite[Lem.~3.5 p.~13]{Celik}, we  directly obtain the
following result of maximal regularity:
\begin{theorem}\label{ThCelik}
Let $n=3$ and $p\in ]1,+\infty[$. Then there exists a unique solution $u\in W^{2}_p(\mathbb{T}_t;L^p(\Omega))\cap W^{1}_p(\mathbb{T}_t;W^2_p(\Omega))$ with the mean value zero
\begin{equation}\label{EqMV0U}
   \int_{\mathbb{T}_t}u(s,x)\;ds=0\quad\forall x\in \Omega
\end{equation}
 of the following periodic boundary value problem
\begin{equation}\label{kuzlininhomper}
\left\lbrace\begin{array}{l}
u_{tt}-c^2 \Delta u - \nu \varepsilon \Delta u_t=f \;\;\;\hbox{ on }\;\mathbb{T}_t\times \Omega,\\
u =g\;\;\;\hbox{ on }\;\mathbb{T}_t\times \partial\Omega\
\end{array}\right.
\end{equation}
if and only if the functions $f$ and $g$ satisfy
\begin{equation}\label{EqfTh51}
 f\in L^p(\mathbb{T}_t;L^p(\Omega)) \hbox{ and } g\in W^{2-\frac{1}{2p}}_p(\mathbb{T}_t;L^p(\partial\Omega))\cap W^1_p(\mathbb{T}_t;W^{2-\frac{1}{p}}_p(\partial\Omega))
\end{equation}
and are  of mean value zero:
\begin{equation}\label{EqMV0}
  \int_{\mathbb{T}_t}f(l,x)\;dl=0\quad \forall x\in \Omega\hbox{ and } \int_{\mathbb{T}_t}g(l,x')\;dl=0\quad\forall x'\in \partial\Omega .
\end{equation}
Moreover,  the following stability estimate  holds
\begin{align*}
\Vert u\Vert_{W^{2}_p(\mathbb{T}_t;L^p(\Omega))\cap W^{1}_p(\mathbb{T}_t;W^2_p(\Omega))}\leq C & \left(\Vert f\Vert_{L^p(\mathbb{T}_t;L^p(\Omega))}\right.\\
&+\left.\Vert g\Vert_{ W^{2-\frac{1}{2p}}_p(\mathbb{T}_t;L^p(\partial\Omega))\cap W^1_p(\mathbb{T}_t;W^{2-\frac{1}{p}}_p(\partial\Omega)) } \right).
\end{align*}
\end{theorem}
\begin{proof}
On one hand,
if $f$ and $g$ satisfy~(\ref{EqfTh51})--(\ref{EqMV0}), the necessity of the conditions is shown in Ref.~\cite{Celik}.
On the other hand, the conditions~(\ref{EqfTh51})--(\ref{EqMV0}) are sufficient by a direct application of the trace theorems recalled in Ref.~\cite{Celik} pp.~6--7 and proved in Ref.~\cite{Denk} Section~3 for example.
\end{proof}
The results of Ref.~\cite{Celik} allow to see that Theorem~\ref{ThCelik} does not depend on~$n$, moreover if we look at the case $p=2$ the linearity of the operator $\partial^2_t-c^2\Delta-\nu \Delta\partial_t$ from Eq.~(\ref{kuzlininhomper}) implies that we can work with $H^s(\Omega)$ instead of $L^2(\Omega)$:
\begin{lemma}\label{thmLinPer}
Let $n\in \mathbb{N}^*$ and $s\geq 0$. There exists a unique solution of the periodic in time boundary value problem for the linear strongly damped wave equation~(\ref{kuzlininhomper})
\begin{equation}\label{Timepersolspac}
u\in X=\left\lbrace u\in H^2(\mathbb{T}_t;H^s(\Omega))\cap H^1(\mathbb{T}_t;H^{s+2}(\Omega))\vert\; \int_{\mathbb{T}_t}u(s,x)\;ds=0\;\forall x\in \Omega \right\rbrace
\end{equation}
if and only if $f$ and $g$ satisfy
\begin{equation}\label{timeperboundreg}
f\in L^2(\mathbb{T}_t;H^s(\Omega))\hbox{ and }g\in \mathbb{F}_{\mathbb{T}}= H^{\frac{7}{4}}(\mathbb{T}_t;H^s(\partial\Omega))\cap H^1(\mathbb{T}_t;H^{s+\frac{3}{2}}(\partial\Omega))
\end{equation}
along with~(\ref{EqMV0}).

Moreover the following stability estimate  holds
$$\Vert u\Vert_{X}\leq C (\Vert f\Vert_{L^2(\mathbb{T}_t;H^s(\Omega))}+\Vert g\Vert_{ \mathbb{F}_{\mathbb{T}}} ).$$
Here $H^2(\mathbb{T}_t;H^s(\Omega))\cap H^1(\mathbb{T}_t;H^{s+2}(\Omega))$ is endowed with its usual norm denoted here and in the sequel by $\Vert.\Vert_X$.
\end{lemma}

 To prove the global well-posedness of the periodic in time boundary
value problem~(\ref{kuzper}) for the Kuznetsov equation we
use its boundary condition as the initial condition of the corresponding Cauchy
problem in $\mathbb{R}^n$ and we combine the maximal regularity result for
system~(\ref{kuzlininhomper}) with~\cite[1.5~Cor.,
p.~368]{Sukhinin} (see also~\cite[Thm.~4.2]{Perso}) applying the
same method as previously done for the Cauchy problem associated with the
Kuznetsov equation~\cite{Perso}.
\begin{theorem}\label{globwellposKuzper}
Let $\nu>0$, $n\in \mathbb{N}^*$ and  $s>\frac{n}{2}$. 
Let
$
X\hbox{ be defined by }(\ref{Timepersolspac})
$
and the boundary condition
$
g\in \mathbb{F}_{\mathbb{T}}$ be defined by~(\ref{timeperboundreg}) and in addition, let $g$ be of  mean value zero (see Eq.~(\ref{EqMV0})).

Then there exist $r^*=O(1)$ and $C_1=O(1)$ such that for all $r\in [0,r^*[$, if
$\Vert g\Vert_{\mathbb{F}_{\mathbb{T}}}\leq \frac{\sqrt{\nu \varepsilon}}{C_1}r,$
there exists a unique solution $u\in X$ of the periodic problem~(\ref{kuzper}) for the Kuznetsov equation such that $\Vert u\Vert_X\leq 2r$.
\end{theorem}

\begin{proof}
For $g\in\mathbb{F}_{\mathbb{T}}$ defined in~(\ref{timeperboundreg}) and satisfying~(\ref{EqMV0}), let us denote by $u^*\in X$ the unique solution of the linear problem~(\ref{kuzlininhomper}) with $f=0$ and $g\in \mathbb{F}_{\mathbb{T}}$.

In addition, according to Theorem~\ref{thmLinPer}, we take $X$ defined in~(\ref{Timepersolspac}),
 this time for $s>\frac{n}{2}$ (we need this regularity to control the non-linear terms),  and introduce the Banach spaces
\begin{equation}
 X_0:=\lbrace u\in X|\;\;\;u\vert_{\partial\Omega}=0\hbox{ on } \mathbb{T}_t\times\partial\Omega \rbrace
\end{equation}
and
$$Y=\left\lbrace f\in L^2(\mathbb{T}_t;H^s(\Omega))\vert\;  \int_{\mathbb{T}_t}f(s,x)\;ds=0\quad\forall x\in \Omega \right\rbrace.$$ Then by Lemma~\ref{thmLinPer}, the linear operator
$$L:X_0\rightarrow Y,\quad  u\in X_0\mapsto\;L(u):=u_{tt}-c^2\Delta u-\nu \varepsilon \Delta u_t\in Y,$$
 is a bi-continuous isomorphism.

 Let us now notice that if $v$ is the unique solution of the non-linear Dirichlet problem
 \begin{equation}\label{SystkuznV}
\left\lbrace
\begin{array}{lll}
v_{tt}-c^2\Delta v-\nu\varepsilon \Delta v_t=&\alpha \varepsilon (v+u^*)_t(v+u^*)_{tt}&\hbox{ on } \mathbb{T}_t\times\Omega ,\\

&+\beta \varepsilon \nabla (v+u^*).\nabla(v+u^*)_t &  \\
v=0\hbox{ on } \mathbb{T}_t\times\partial\Omega ,& &\end{array}
\right.
\end{equation}
 then $u=v+u^*$ is the unique solution of the periodic problem~(\ref{kuzper}).
 Let us prove the existence of a such $v$,
using~\cite[1.5~Cor., p.~368]{Sukhinin}.

We suppose that $\Vert u^*\Vert_X\leq r$
and define for $v\in X_0$
$$\Phi(v):=\alpha \varepsilon (v+u^*)_t(v+u^*)_{tt}+\beta \varepsilon \nabla (v+u^*).\nabla(v+u^*)_t.$$

For $w$ and $z$ in $X_0$ such that
$\Vert w\Vert_X\leq r$ and $\Vert z\Vert_X\leq r$,
 we estimate the norm $\Vert \Phi(w)-\Phi(z)\Vert_Y $.
By applying the triangular inequality we have
\begin{multline*}
\Vert \Phi(w)-\Phi(z)\Vert_Y\leq  \alpha \varepsilon \Big(\Vert u^*_t (w-z)_{tt} \Vert_Y+\Vert (w-z)_t u^*_{tt}\Vert_Y\\
+\Vert  w_t (w-z)_{tt}\Vert_Y+\Vert (w-z)_t z_{tt}\Vert_Y\Big)\\
+\beta \varepsilon\Big( \Vert \nabla u^* \nabla(w-z)_t \Vert_Y+\Vert \nabla (w-z)\nabla u^*_t \Vert_Y\\
+\Vert \nabla w\nabla (w-z)_t \Vert_Y+\Vert \nabla (w-z) \nabla z_t \Vert_Y\Big).
\end{multline*}
 Now, for all $a$ and $b$ in $X$ with $s\ge s_0> \frac{n}{2}$  it holds
 \begin{align*}
 \Vert a_t b_{tt}\Vert_Y\leq & \Vert a_t \Vert_{L^\infty(\mathbb{T}_t\times\Omega)} \Vert b_{tt}\Vert_Y\\
 \leq & C_{H^1(\mathbb{T}_t;H^{s_0}(\Omega))\to L^\infty(\mathbb{T}_t\times\Omega)} \Vert a_t\Vert_{H^1(\mathbb{T}_t;H^{s_0}(\Omega))} \Vert b\Vert_{X}\\
 \leq & C_{H^1(\mathbb{T}_t;H^{s_0}(\Omega))\to L^\infty(\mathbb{T}_t\times\Omega)} \Vert a\Vert_{X} \Vert b\Vert_{X},
\end{align*}
where $C_{H^1(\mathbb{T}_t;H^{s_0}(\Omega))\to L^\infty(\mathbb{T}_t\times\Omega)}$ is the embedding constant of $H^1(\mathbb{T}_t;H^{s_0}(\Omega))$ in $L^\infty(\mathbb{T}_t\times\Omega)$, independent of $s$, but depending only on the dimension $n$.
In the same way, for all $a$ and $b$ in $X$ it holds
$$\Vert \nabla a \nabla  b_t\Vert_Y\leq C_{H^1(\mathbb{T}_t;H^{s_0}(\Omega))\to L^\infty(\mathbb{T}_t\times\Omega)} \Vert a\Vert_{X} \Vert b\Vert_{X}.$$
Taking $a$ and $b$ equal to $u^*$, $w$, $z$ or $w-z$, as $\Vert u^*\Vert_X\leq r$, $\Vert w\Vert_X\leq r$ and $\Vert z\Vert_X\leq r$, we obtain
\begin{align*}
\Vert \Phi(w)-\Phi(z)\Vert_Y\leq %
4 (\alpha+\beta)C_{H^1(\mathbb{T}_t;H^{s_0}(\Omega))\to L^\infty(\mathbb{T}_t\times\Omega)} \varepsilon r \Vert w-z\Vert_X.
\end{align*}
By the fact that $L$ is a bi-continuous isomorphism, there exists a minimal constant $C_\eps=O\left(\frac{1}{\eps \nu} \right)>0$, coming from the inequality $$C_0 \eps \nu\|u\|_X^2\le \|f\|_Y\|u\|_X$$ for $u$, a solution of the linear problem~(\ref{kuzlininhomper}) with homogeneous boundary data (for a maximal constant $C_0=O(1)>0$)
such that
$$  \Vert u\Vert_X\leq C_\eps \Vert Lu\Vert_Y \quad \forall u\in X_0.$$
Hence, for all $f\in Y$
$$P_{LU_{X_0}}(f)\leq C_\eps P_{U_Y}(f)=C_\eps\Vert f\Vert_Y.$$
Then we find for $w$ and $z$ in $X_0$, such that $\|w\|_X\le r$, $\|z\|_X\le r$, and also for $\|u^*\|_X\le r$, that with the notation $$\Theta(r):= 4 C_\eps (\alpha+\beta)C_{H^1(\mathbb{T}_t;H^{s_0}(\Omega))\to L^\infty(\mathbb{T}_t\times\Omega)}\varepsilon r$$ it holds
$$P_{LU_{X_0}}(\Phi(w)-\Phi(z))\leq \Theta(r) \Vert w-z\Vert_X.$$
Thus we apply~\cite[1.5~Cor., p.~368]{Sukhinin}  with
$f(x)=L(x)-\Phi(x)$ and $x_0=0$. Therefore, knowing that $C_\eps=\frac{C_0}{\eps
\nu}$, we have, that for all  $r\in[0,r_{*}[$ with
\begin{equation}\label{Eqret}
 r_{*}=\frac{\nu}{4 C_0 (\alpha+\beta)C_{H^1(\mathbb{T}_t;H^{s_0}(\Omega))\to L^\infty(\mathbb{T}_t\times\Omega)}}=O(1),
\end{equation}
  for all $y\in \Phi(0)+w(r) L U_{X_0}\subset Y$
with $$w(r)= r-2 \frac{C_0}{\nu} C_{H^1(\mathbb{T}_t;H^{s_0}(\Omega))\to L^\infty(\mathbb{T}_t\times\Omega)} (\alpha+\beta) r^2,$$
there exists a unique $v\in 0+r U_{X_0}$ such that $L(v)-\Phi(v)=y$.
Since we are seeking  $v$, which solves the non-linear
problem~(\ref{SystkuznV}), we need to impose $y=0$, $i.e.$ 
that $v$ be the solution of the non-linear  problem~(\ref{SystkuznV}), then we
need to impose $y=0$ and thus,
to ensure that $$0\in \Phi(0)+w(r) L U_{X_0}.$$
Since $-\frac{1}{w(r)}\Phi(0)$ is an element of $Y$ and $LX_0=Y$, there exists a unique $z\in X_0$ such that
\begin{equation}\label{Eqz}
 L z=-\frac{1}{w(r)}\Phi(0).
\end{equation}
Let us show that $\|z\|_X\le 1$, what will implies that $0\in \Phi(0)+w(r) L U_{X_0}$.
Noticing that
\begin{align*}
\Vert \Phi(0)\Vert_Y & \leq \alpha \varepsilon \Vert v_t v_{tt}\Vert_Y +\beta \varepsilon \Vert \nabla v \nabla v_t\Vert_Y\\
& \leq  (\alpha+\beta) \varepsilon C_{H^1(\mathbb{T}_t;H^{s_0}(\Omega))\to L^\infty(\mathbb{T}_t\times\Omega)}\Vert v\Vert_X^2 \\
& \leq (\alpha+\beta) \varepsilon C_{H^1(\mathbb{T}_t;H^{s_0}(\Omega))\to L^\infty(\mathbb{T}_t\times\Omega)}r^2
\end{align*}
and using~(\ref{Eqz}) we find
\begin{align*}
 & \Vert z\Vert_X \leq C_\eps\Vert L z\Vert_Y=C_\eps\frac{\Vert \Phi(0)\Vert_Y}{w(r)}\\
 &\leq \frac{C_\eps C_{H^1(\mathbb{T}_t;H^{s_0}(\Omega))\to L^\infty(\mathbb{T}_t\times\Omega)} (\alpha+\beta) \varepsilon r}{(1-2 C_\eps C_{H^1(\mathbb{T}_t;H^{s_0}(\Omega))\to L^\infty(\mathbb{T}_t\times\Omega)} (\alpha+\beta)\varepsilon r)}<\frac{1}{2},
\end{align*}
as soon as $r<r^*$.

Consequently, $z\in U_{X_0}$ and $\Phi(0)+w(r) Lz=0$.
Then we conclude that  for all  $r\in[0,r_{*}[$, if $\|u^*\|_X\le r$, there
exists a unique $v\in r U_{X_0}$ such that $L(v)-\Phi(v)=0$, $i.e.$
 $v$ is the solution of the non-linear problem~(\ref{SystkuznV}).
Thanks to the maximal regularity and a priori estimate following from Theorem~\ref{thmLinPer} with $f=0$,
there exists a constant $C_1=O(\eps^0)>0$, such that
$$\|u^*\|_X\le \frac{C_1}{\sqrt{\nu \eps}}\Vert g\Vert_{\mathbb{F}_{\mathbb{T}}}.$$
Thus, for all  $r\in[0,r_{*}[$ and $\Vert g\Vert_{\mathbb{F}_{\mathbb{T}}}\le \frac{\sqrt{\nu \eps}}{C_1}r$, the function $u=u^*+v\in X$ is the unique solution of the time periodic problem for the Kuznetsov equation and $\Vert u\Vert_X\leq 2 r$.
\end{proof}

\subsection{Initial boundary value problem.}\label{WPKuzhalf2}
We still work on $\Omega=\mathbb{R}_+\times \mathbb{R}^{n-1}$ and we study the
initial boundary value problem for the Kuznetsov equation on this space,
\textit{i.e.} the perturbation of an imposed initial condition by a source on
the boundary, which in Subsection~\ref{sssKKZKperbcIn} was
determined by the solution of the KZK equation.
\begin{lemma}\label{maxregstrdamphalf}
Let $s\geq 0$, $n\in \mathbb{N}$. There exists a unique solution
\begin{equation}\label{Halfspaceregsol}
u\in \mathbb{E}:=H^2(\mathbb{R_+};H^s(\Omega))\cap H^1( \mathbb{R_+};H^{s+2}(\Omega))
\end{equation}
of the linear problem
\begin{equation}\label{eqlinhalf}
\left\lbrace
\begin{array}{c}
u_{tt}-c^2 \Delta u -\nu \varepsilon \Delta u_t=f\;\;\;\hbox{ in }\;\mathbb{R_+}\times \Omega,\\
u=g\;\;\;\hbox{ on }\; \mathbb{R_+}\times\partial\Omega,\\
u(0)=u_0, \;\;\;u_t(0)=u_1\;\;\;\hbox{ in }\; \Omega
\end{array}\right.
\end{equation}
if and only if the data satisfy the following conditions
\begin{itemize}
\item $f\in L^2(\mathbb{R}_+;H^s(\Omega)),$
\item for the boundary condition
\begin{equation}\label{Halfspaceregbound}
g\in \mathbb{F}_{\mathbb{R}_+}= H^{7/4}(\mathbb{R}_+;H^s(\partial\Omega))\cap
H^1(\mathbb{R}_+;H^{s+3/2}(\partial\Omega)),
\end{equation}
\item $u_0\in H^{s+2}(\Omega)$ and $u_1\in H^{s+1}(\Omega)$,
\item $g(0)=u_0$ and $g_t(0)=u_1$ on $\partial\Omega $ in the trace sense.
\end{itemize}
In addition, the solution satisfies the stability estimate
$$\Vert u\Vert_{\mathbb{E}}\leq C (\Vert f\Vert_{L^2(\mathbb{R}_+;H^s(\Omega))}+\Vert g\Vert_{\mathbb{F}_{\mathbb{R}_+}}+\Vert u_0\Vert_{H^{s+2}}+\Vert u_1\Vert_{H^{s+1}}).$$
\end{lemma}
In order to prove this result we will use the following lemma to remove the inhomogeneity $g$.
\begin{lemma}\label{lemwpheat}
Let $s\geq 0$, $n\in \mathbb{N}$ and $\mathbb{E}$ defined in~(\ref{Halfspaceregsol}). There exists a unique solution $ w\in \mathbb{E}$ of the following linear problem
\begin{equation}\label{heatderiv}
\left\lbrace
\begin{array}{c}
w_{tt} -\nu \varepsilon \Delta w_t=0\;\;\;\hbox{ in }\;\mathbb{R_+}\times \Omega,\\w=g\;\;\;\hbox{ on }\; \mathbb{R_+}\times\partial\Omega,\\
w(0)=0, \;\;\;w_t(0)=0\;\;\;\hbox{ in }\; \Omega
\end{array}\right.
\end{equation}
if and only if

$g\in \mathbb{F}_{\mathbb{R}_+}$ (the space $\mathbb{F}_{\mathbb{R}_+}$ is defined in~(\ref{Halfspaceregbound})) and it holds the following compatibility conditions:

for all $x\in\partial\Omega$, $g(0)=0$ and $g_t(0)=0$.

Moreover, the solution $w$ satisfies the stability estimate
$$\Vert w\Vert_{\mathbb{E}}\leq C \Vert g\Vert_{\mathbb{F}_{\mathbb{R}_+}}.$$
\end{lemma}

\begin{proof}
First we prove the sufficiency. By assumption~(\ref{Halfspaceregbound}), we have
$$\partial_t g\in H^{3/4}(\mathbb{R}_+;H^s(\partial\Omega))\cap L^2(\mathbb{R}_+;H^{s+3/2}(\partial\Omega)).$$
Thanks to $\S$~3 in Ref.~\cite[p.~288]{Ladyzhenskaya}, we obtain
a unique solution
$$v\in H^1(\mathbb{R_+};H^s(\Omega))\cap L^2( \mathbb{R_+};H^{s+2}(\Omega))$$
of the parabolic problem
$$v_t-\nu \varepsilon \Delta v=0\;\hbox{ in }\;\mathbb{R}_+\times \Omega,\;v=\partial_t g\;\hbox{ on }\;\mathbb{R}_+\times\partial\Omega,\;v(0)=0\;\hbox{ in }\;\Omega.$$
Next we define for $t\in\mathbb{R}_+$ and $x\in\Omega$ the function
$$w(t,x):=\int_0^t v(l,x)dl. $$ 
We have $w(0)=0$ and $w_t(0)=0$. Moreover, it satisfies
$$w_{tt}-\nu\varepsilon \Delta w_t=0, \quad
w(t)\vert_{\partial\Omega}=\int_0^t g_t(l)\;dl=g(t),$$
as $g(0)=0$. Therefore, $w$ is a solution of problem~(\ref{heatderiv}). The necessity follows from the spatial trace theorem ensuring that the trace operator $Tr_{\partial\Omega}:u  \mapsto u\vert_{\partial\Omega}$, considering as a map
\begin{align}
&H^1(\mathbb{R_+};H^s(\Omega))\cap L^2( \mathbb{R_+};H^{s+2}(\Omega))  \rightarrow H^{3/4}(\mathbb{R}_+;H^s(\partial\Omega))\cap L^2(\mathbb{R}_+;H^{s+3/2}(\partial\Omega)),\label{Trachalf1}
\end{align}
is bounded and surjective by~\cite[Lem.~3.5]{Denk}.
For the compatibility condition, thanks to~\cite[Lem.~11]{Blasio},
we also know that the temporal trace $Tr_{t=0}:g\mapsto~g\vert_{t=0}$,
considered as a map
\begin{align}
&H^{3/4}(\mathbb{R}_+;H^s(\partial\Omega))\cap L^2(\mathbb{R}_+;H^{s+3/2}(\partial\Omega)) \rightarrow H^{s+1/2}(\partial\Omega),
\label{Trachalf2}
\end{align}
is well defined and bounded. Moreover, the spatial trace
\begin{equation}\label{Trachalf3}
H^{s+1/2}(\Omega)\rightarrow H^s(\partial\Omega)
\end{equation}
is bounded by~\cite[Thm.~1.5.1.1]{Grisvard}.

To obtain uniqueness, let $w$ be a solution to~(\ref{heatderiv}) with $g=0$. Since $w_t$ solves the heat problem with homogeneous data, we obtain $w_t=0$ and therefore also $w=0$ by the initial condition $w(0)=0$.
The stability estimate follows from the closed graph theorem.
\end{proof}
Let us prove Lemma~\ref{maxregstrdamphalf}:
\begin{proof}
We obtain the uniqueness of the solution of the boundary value problem for the linear strongly damped equation~(\ref{eqlinhalf}) from the fact that in the case $g=0$ we can consider $-\Delta$ as a self-adjoint and non negative operator with homogeneous Dirichlet boundary conditions and we can use~\cite{Haraux}.

To verify the necessity of the conditions on the data, we suppose that $u\in\mathbb{E}$ (see Eq.~(\ref{Halfspaceregsol}) for the definition of $\mathbb{E}$) is a solution of~(\ref{eqlinhalf}).
Then
$$u,\;u_t\in  H^1(\mathbb{R_+};H^s(\Omega))\cap L^2( \mathbb{R_+};H^{s+2}(\Omega)) \hbox{ and thus } f\in L^2(\mathbb{R}_+;H^s(\Omega)).$$
Taking as in the previous proof the spatial trace $Tr_{\partial\Omega}$ as in Eq.~(\ref{Trachalf1}) we have
$$g,\;g_t\in H^{3/4}(\mathbb{R}_+;H^s(\partial\Omega))\cap L^2(\mathbb{R}_+;H^{s+3/2}(\partial\Omega)), \hbox{ which implies }g\in \mathbb{F}_{\mathbb{R}_+}.$$
By the Sobolev embedding $H^1(\mathbb{R}_+;H^{s+2}(\Omega))\hookrightarrow C(\mathbb{R}_+;H^{s+2}(\Omega))$, it follows that $u_0\in H^{s+2}(\Omega)$ and we also have the temporal trace
$$u\mapsto u\vert_{t=0}:H^1(\mathbb{R_+};H^s(\Omega))\cap L^2( \mathbb{R_+};H^{s+2}(\Omega))\rightarrow H^{s+1}(\Omega)$$
by~\cite[Lem.~3.7]{Denk}. Following the proof of
Lemma~\ref{lemwpheat}, we use Eqs.~(\ref{Trachalf2})
and~(\ref{Trachalf3}) to obtain the compatibility conditions.

It remains to prove the sufficiency of the conditions. We extend $u_0$, $u_1$ and $f$ in odd functions among $x_1$ on $\mathbb{R}^n$ so that we have
$$\tilde{u}_0\in H^{s+2}(\mathbb{R}^n),\; \tilde{u}_1\in H^{s+1}(\mathbb{R}^n)
\; \hbox{ and }\;  \tilde{f}\in L^2(\mathbb{R}_+;H^s(\mathbb{R}^n)).$$
 Considering the non homogeneous linear Cauchy problem
$$
\left\lbrace
\begin{array}{l}
\tilde{u}_{tt}-c^2\Delta \tilde{u} -\nu \varepsilon \Delta \tilde{u}_t=\tilde{f}\;\;\;\hbox{ in }\;\mathbb{R_+}\times \mathbb{R}^n,\\
\tilde{u}(0)=\tilde{u}_0, \;\;\;\tilde{u}_t(0)=\tilde{u}_0\;\;\;\hbox{ in }\;
\mathbb{R}^n,
\end{array}\right.$$
by~\cite[Thm.~4.1]{Perso} we obtain the
existence of its unique solution
$$\tilde{u}\in H^2(\mathbb{R_+};H^s(\mathbb{R}^n))\cap H^1( \mathbb{R_+};H^{s+2}(\mathbb{R}^n)).$$
Let $\overline{u}\in \mathbb{E}$, denote the restriction of $\tilde{u}$ to $\Omega$ and let $\overline{g}:= g-\overline{u}\vert_{\partial\Omega}$.
By the spatial trace theorem $\overline{u}\vert_{\partial\Omega}\in \mathbb{F}_{\mathbb{R}_+}$, and hence $\overline{g}\in  \mathbb{F}_{\mathbb{R}_+}$. Then the solution $u$ of the non homogeneous linear problem~(\ref{eqlinhalf}) is given by $u=v+\overline{u}$, where $v$ solves  problem~(\ref{eqlinhalf}) with $f=u_0=u_1=0$ and $g=\overline{g}$. %
>From Lemma~\ref{lemwpheat} we have a unique solution $\overline{v}\in \mathbb{E}_u$ of problem~(\ref{heatderiv}) with $g=\overline{g}$.
Then the function $w:=v-\overline{v}$ solves the following system
$$
\left\lbrace
\begin{array}{l}
w_{tt}-\Delta w -\nu \varepsilon \Delta w_t= c^2 \Delta \overline{v}\;\;\;\hbox{ in }\;\mathbb{R_+}\times \Omega,\\
w=0\;\;\;\hbox{ on }\; \mathbb{R_+}\times\partial\Omega,\\
w(0)=0, \;\;\;w_t(0)=0\;\;\;\hbox{ in }\; \Omega,
\end{array}\right.
$$
which thanks to~\cite[Thm.~2.6]{Haraux} has a
unique solution $w\in \mathbb{E}$ defined in~(\ref{Halfspaceregsol}). The
function $u:=w+\overline{v}+\overline{u}$ is the desired solution of
system~(\ref{eqlinhalf}) and the stability estimate follows from the closed
graph theorem. This concludes the proof of Lemma~\ref{maxregstrdamphalf}.
\end{proof}

 The next theorem follows from  the maximal regularity result of
Lemma~\ref{maxregstrdamphalf} and
of~\cite[1.5.~Cor.,~p.~368]{Sukhinin}. Its
proof is similar to the proof of Theorem~\ref{globwellposKuzper} and hence is
omitted.
\begin{theorem}\label{ThMainWPnuPGlobHalf}
 Let $\nu>0$, $n\in \mathbb{N}^*$, $\Omega=\mathbb{R}_+\times \mathbb{R}^{n-1}$ and $s>\frac{n}{2}$. Considering the initial boundary value problem for the Kuznetsov equation in the half space with the Dirichlet boundary condition~(\ref{kuzhalfinitbound})
the following results hold: there exist constants $r^*=O(1)$ and $C_1=O(1)$, such that for all  initial data satisfying 
\begin{itemize}
\item $g\in \mathbb{F}_{\mathbb{R}^+}:= H^{7/4}([0,\infty[;H^s(\partial\Omega))\cap H^1([0,\infty[;H^{s+3/2}(\partial\Omega))$,
\item $u_0\in H^{s+2}(\Omega)$, $u_1\in H^{s+1}(\Omega)$,
\item $g(0)=u_0\vert_{\partial\Omega}$ and $g_t(0)=u_1\vert_{\partial\Omega}$,
\end{itemize} and such that for $r\in [0,r^*[$
\begin{equation*}
 \Vert u_0\Vert_{H^{s+2}(\Omega)}+\Vert u_1\Vert_{H^{s+1}(\Omega)}+\Vert g\Vert_{\mathbb{F}_{[0,T]}}\leq \frac{\nu \varepsilon}{C_1}r,
\end{equation*}
 there exists  a unique solution of problem~(\ref{kuzhalfinitbound}) for the Kuznetsov equation
  $$u\in H^2([0,\infty[;H^s(\Omega))\cap H^1( [0,\infty[;H^{s+2}(\Omega)), $$
  such that
  $$ \Vert u\Vert_{H^2([0,\infty[;H^s(\Omega))\cap H^1( [0,\infty[;H^{s+2}(\Omega))}\leq 2r.$$
  \end{theorem}

\medskip
Received May 2019; 1st revision October 2019; 2nd revision January 2020.
\medskip


\begin{thebibliography}{99}
\bibitem{Aanonsen}
\newblock {S.~I. Aanonsen, T.~Barkve, J.~N. Tj\o{}tta and S.~Tj\o{}tta},
\newblock {Distortion and harmonic generation in the nearfield of a finite amplitude sound beam},
\newblock \emph{The Journal of the Acoustical Society of America}, \textbf{75} (1984), 749--768.

\bibitem{Adams} 
\newblock {R.~A. Adams},
\newblock \emph{Sobolev Spaces},
\newblock Vol. 65, Pure and Applied Mathematics, Academic Press [A subsidiary of Harcourt Brace Jovanovich, Publishers], New York-London, 1975.

\bibitem{Alinhac2}
\newblock {S.~Alinhac},
\newblock {Temps de vie des solutions r\'{e}guli\`{e}res des \'{e}quations d'{E}uler compressibles axisym\'{e}triques en dimension deux},
\newblock \emph{Invent. Math.}, \textbf{111} (1993), 627--670.

\bibitem{Alinhac} 
\newblock {S.~Alinhac},
\newblock {A minicourse on global existence and blowup of classical solutions to multidimensional quasilinear wave equations},
\newblock \emph{Journ\'{e}es ``\'{E}quations aux {D}\'eriv\'ees {P}artielles'' ({F}orges-les-{E}aux, 2002)}, (2002), Exp. No. I, 33 pp.

\bibitem{Bakhvalov} 
\newblock {N.~S. Bakhvalov, Y.~M. Zhile\u{i}kin and E.~A. Zabolotskaya},
\newblock \emph{Nonlinear Theory of Sound Beams},
\newblock American Institute of Physics Translation Series, American Institute of Physics, New York, 1987.

\bibitem{Bers}
\newblock {L.~Bers, F.~John and M.~Schechter},
\newblock \emph{Partial Differential Equations},
\newblock Lectures in Applied Mathematics, 3A, American Mathematical Society, Providence, RI, 1979.

\bibitem{Kuperman2}
\newblock {P. Caine and M. West},
\newblock {A tutorial on the non-linear progressive wave equation (NPE). Part 2. Derivation of the three-dimensional cartesian version without use of perturbation expansions},
\newblock \emph{Applied Acoustics}, \textbf{45} (1995), 155--165.

\bibitem{Cao} 
\newblock {Z.~Cao, H.~Yin, L.~Zhang and L.~Zhu},
\newblock {Large time asymptotic behavior of the compressible {N}avier-{S}tokes equations in partial space-periodic domains},
\newblock \emph{Acta Math. Sci. Ser. B (Engl. Ed.)}, \textbf{36} (2016), 1167--1191.

\bibitem{Celik} 
\newblock {A.~Celik and M.~Kyed},
\newblock {Nonlinear wave equation with damping: Periodic forcing and non-resonant solutions to the {K}uznetsov equation},
\newblock \emph{ZAMM Z. Angew. Math. Mech.}, \textbf{98} (2018), 412--430.

\bibitem{Dafermos} 
\newblock {C.~M. Dafermos},
\newblock \emph{Hyperbolic Conservation Laws in Continuum Physics},
\newblock Vol.~325, Grundlehren der Mathematischen Wissenschaften [Fundamental Principles of Mathematical Sciences], 4$^{th}$ edition, Springer-Verlag, Berlin, 2016.

\bibitem{Perso} 
\newblock {A.~Dekkers and A.~Rozanova-Pierrat},
\newblock {Cauchy problem for the {K}uznetsov equation},
\newblock \emph{Discrete Contin. Dyn. Syst.}, \textbf{39} (2019), 277--307.

\bibitem{Dekkers}
\newblock {A.~Dekkers and A.~Rozanova-Pierrat},
\newblock {Models of nonlinear acoustics viewed as an approximation of the {N}avier-{S}tokes and {E}uler compressible isentropicsystems}, preprint,
\newblock (1811.10850).

\bibitem{Denk} 
\newblock {R.~Denk, M.~Hieber and J.~Pr\"{u}ss},
\newblock {Optimal {$L^p$}-{$L^q$}-estimates for parabolic boundary value problems with inhomogeneous data},
\newblock \emph{Math. Z.}, \textbf{257} (2007), 193--224.

\bibitem{Blasio} 
\newblock {G.~Di~Blasio},
\newblock {Linear parabolic evolution equations in {$L^p$}-spaces},
\newblock \emph{Ann. Mat. Pura Appl.}, \textbf{138} (1984), 55--104.

\bibitem{Haraux} 
\newblock {M.~Ghisi, M.~Gobbino and A.~Haraux},
\newblock {Local and global smoothing effects for some linear hyperbolic equations with a strong dissipation},
\newblock \emph{Trans. Amer. Math. Soc.}, \textbf{368} (2016), 2039--2079.

\bibitem{Grisvard} 
\newblock {P.~Grisvard},
\newblock \emph{Elliptic Problems in Nonsmooth Domains},
\newblock Vol.~24, Monographs and Studies in Mathematics, Pitman (Advanced Publishing Program), Boston, MA, 1985.

\bibitem{Gustafsson} 
\newblock {B.~Gustafsson and A.~Sundstr\"{o}m},
\newblock {Incompletely parabolic problems in fluid dynamics},
\newblock \emph{SIAM J. Appl. Math.}, \textbf{35} (1978), 343--357.

\bibitem{Hamilton}
\newblock {M.~F. Hamilton and D.~ T. Blackstock},
\newblock \emph{Nonlinear Acoustics},
\newblock Academic Press, 1998.

\bibitem{Hoff1} 
\newblock {D. Hoff},
\newblock {Strong convergence to global solutions for multidimensional flows of compressible, viscous fluids with polytropic equations of state and discontinuous initial data},
\newblock \emph{Arch. Rational Mech. Anal.}, \textbf{132} (1995), 1--14.

\bibitem{Hoff2} 
\newblock {D. Hoff},
\newblock {Discontinuous solutions of the Navier-Stokes equations for multidimensional flows of heat-conducting fluids},
\newblock \emph{Arch. Rational Mech. Anal.}, \textbf{139} (1997), 303--354.

\bibitem{Ito} 
\newblock {K.~Ito},
\newblock {Smooth global solutions of the two-dimensional {B}urgers equation},
\newblock \emph{Canad. Appl. Math. Quart.}, \textbf{2} (1994), 283--323.

\bibitem{John} 
\newblock {F.~John},
\newblock \emph{Nonlinear Wave Equations, Formation of Singularities},
\newblock Vol.~2, University Lecture Series, American Mathematical Society, Providence, RI, 1990.

\bibitem{Jordan} 
\newblock {P.~M. Jordan},
\newblock {An analytical study of {K}uznetsov's equation: Diffusive solitons, shock formation, and solution bifurcation},
\newblock \emph{Phys. Lett. A}, \textbf{326} (2004), 77--84.

\bibitem{Jordan-2014} 
\newblock {P.~M. Jordan},
\newblock {Second-sound phenomena in inviscid, thermally relaxing gases},
\newblock \emph{Discrete Contin. Dyn. Syst. Ser. B}, \textbf{19} (2014), 2189--2205.

\bibitem{Kalt2} 
\newblock {B.~Kaltenbacher and I.~Lasiecka},
\newblock {Well-posedness of the {W}estervelt and the {K}uznetsov equation with nonhomogeneous {N}eumann boundary conditions},
\newblock \emph{Discrete Contin. Dyn. Syst.}, (2011), 763--773.

\bibitem{Kaltenbacher1} 
\newblock {B.~Kaltenbacher, I.~Lasiecka and M. K. Pospieszalska},
\newblock {Well-posedness and exponential decay of the energy in the nonlinear {J}ordan-{M}oore-{G}ibson-{T}hompson equation arising in high intensity ultrasound},
\newblock \emph{Math. Models Methods Appl. Sci.}, \textbf{22} (2012), 34 pp.

\bibitem{Kaltenbacher2} 
\newblock {B.~Kaltenbacher and V. Nikoli\'{c}},
\newblock {The {J}ordan-{M}oore-{G}ibson-{T}hompson equation: Well-posedness with quadratic gradient nonlinearity and singular limit for vanishing relaxation time},
\newblock \emph{Math. Models Methods Appl. Sci.}, \textbf{29} (2019), 2523--2556.

\bibitem{Barbara} 
\newblock {B. Kaltenbacher and M. Thalhammer},
\newblock {Fundamental models in nonlinear acoustics part I. Analytical comparison},
\newblock \emph{Math. Models Methods Appl. Sci.}, \textbf{28} (2018), 2403--2455.

\bibitem{Kalt1} 
\newblock {B. Kaltenbacher and I. Lasiecka},
\newblock {An analysis of nonhomogeneous {K}uznetsov's equation: Local and global well-posedness; Exponential decay},
\newblock \emph{Math. Nachr.}, \textbf{285} (2012), 295--321.

\bibitem{Kato1} 
\newblock {T.~Kato},
\newblock {The {C}auchy problem for quasi-linear symmetric hyperbolic systems},
\newblock \emph{Arch. Rational Mech. Anal.}, \textbf{58} (1975), 181--205.

\bibitem{Kuznetsov}
\newblock {V. P.~Kuznetsov},
\newblock {Equations of nonlinear acoustics},
\newblock \emph{Soviet Phys. Acoust.}, \textbf{16} (1971), 467--470.

\bibitem{Ladyzhenskaya} 
\newblock {O.~A. Lady\v{z}enskaja, V.~A. Solonnikov and N.~N. Ural'ceva},
\newblock \emph{Linear and Quasilinear Equations of Parabolic Type},
\newblock American Mathematical Society, Providence, RI, 1968.

\bibitem{Lesser}
\newblock {M. J.~Lesser and R.~Seebass},
\newblock {The structure of a weak shock wave undergoing reflexion from a wall},
\newblock \emph{Journal of Fluid Mechanics}, \textbf{31} (1968), 501--528.

\bibitem{LUO-2016}
\newblock {T. Luo, C. Xie and Z. Xin},
\newblock {Non-uniqueness of admissible weak solutions to compressible Euler systems with source terms},
\newblock \emph{Adv. Math.}, \textbf{291} (2016), 542--583.

\bibitem{Makarov}
\newblock {S.~Makarov and~ M. Ochmann},
\newblock {Nonlinear and thermoviscous phenomena in acoustics, part II},
\newblock \emph{Acta Acustica United with Acustica}, \textbf{83} (1997), 197--222.

\bibitem{Matsumura2} 
\newblock {A.~Matsumura and T.~Nishida},
\newblock {Initial-boundary value problems for the equations of motion of compressible viscous and heat-conductive fluids},
\newblock \emph{Comm. Math. Phys.}, \textbf{89} (1983), 445--464.

\bibitem{Matsumura} 
\newblock {A.~Matsumura and T.~Nishida},
\newblock {The initial value problem for the equations of motion of viscous and heat-conductive gases},
\newblock \emph{J. Math. Kyoto Univ.}, \textbf{20} (1980), 67--104.

\bibitem{Kuperman1}
\newblock {B. E. McDonald, P. Caine and M. West},
\newblock {A tutorial on the nonlinear progressive wave equation (NPE)--part 1},
\newblock \emph{Applied Acoustics}, \textbf{43} (1994), 159--167.

\bibitem{McDonald} 
\newblock {B.~E. McDonald and W.~A. Kuperman},
\newblock {Time-domain solution of the parabolic equation including nonlinearity},
\newblock \emph{Comput. Math. Appl.}, \textbf{11} (1985), 843--851.

\bibitem{Meyer}
\newblock {S.~Meyer and M.~Wilke},
\newblock {Global well-posedness and exponential stability for {K}uznetsov's equation in {$L_p$}-spaces},
\newblock \emph{Evol. Equ. Control Theory}, \textbf{2} (2013), 365--378.

\bibitem{Roz2} 
\newblock {A.~Rozanova-Pierrat},
\newblock {Qualitative analysis of the {K}hokhlov-{Z}abolotskaya-{K}uznetsov ({KZK}) equation},
\newblock \emph{Math. Models Methods Appl. Sci.}, \textbf{18} (2008), 781--812.

\bibitem{Roz3} 
\newblock {A.~Rozanova-Pierrat},
\newblock {On the derivation of the {K}hokhlov-{Z}abolotskaya-{K}uznetsov ({KZK}) equation and validation of the {KZK}-approximation for viscous and non-viscous thermo-elastic media},
\newblock \emph{Commun. Math. Sci.}, \textbf{7} (2009), 679--718.

\bibitem{Roz1}
\newblock {A.~Rozanova-Pierrat},
\newblock {Approximation of a compressible Navier-Stokes system by non-linear acoustical models},
\newblock \emph{Proceedings of the International Conference DAYS on DIFFRACTION}, (2015), 270--276.

\bibitem{Sideris2} 
\newblock {T.~C. Sideris},
\newblock {Formation of singularities in three-dimensional compressible fluids},
\newblock \emph{Comm. Math. Phys.}, \textbf{101} (1985), 475--485.

\bibitem{Sideris3} 
\newblock {T.~C. Sideris},
\newblock {The lifespan of smooth solutions to the three-dimensional compressible {E}uler equations and the incompressible limit},
\newblock \emph{Indiana Univ. Math. J.}, \textbf{40} (1991), 535--550.

\bibitem{Sideris1}
\newblock {T.~C. Sideris},
\newblock {The lifespan of {$3$}{D} compressible flow},
\newblock \emph{S\'{e}minaire sur les \'{E}quations aux {D}\'{e}riv\'{e}es {P}artielles}, \textbf{5} (1992), 10 pp.

\bibitem{Sideris2d}
\newblock {T.~C. Sideris},
\newblock {Delayed singularity formation in {$2$}{D} compressible flow},
\newblock \emph{Amer. J. Math.}, \textbf{119} (1997), 371--422.

\bibitem{Sukhinin}
\newblock {M.~F. Sukhinin},
\newblock {On the solvability of the nonlinear stationary transport equation},
\newblock \emph{Teoret. Mat. Fiz.}, \textbf{103} (1995), 23--31.

\bibitem{TJO}
\newblock {J.~N. {T}j\o{}tta and S.~{T}j\o{}tta},
\newblock {{N}onlinear equations of acoustics, with application to parametric acoustic arrays},
\newblock \emph{{T}he {J}ournal of the {A}coustical {S}ociety of {A}merica}, \textbf{69} (1981), 1644--1652.

\bibitem{Westervelt}
\newblock {P.~J. Westervelt},
\newblock {Parametric acoustic array},
\newblock \emph{The Journal of the Acoustical Society of America}, \textbf{35} (1963), 535--537.

\bibitem{Yin1}
\newblock {H.~Yin and Q.~Qiu},
\newblock {The lifespan for {$3$}-{D} spherically symmetric compressible {E}uler equations},
\newblock \emph{Acta Math. Sinica (N.S.)}, \textbf{14} (1998), 527--534.


\end{thebibliography}
\end{document}